\documentclass[11pt]{amsart}
\usepackage{amsmath}
\usepackage{amssymb}
\usepackage{esint}
\usepackage{tikz}
\usepackage{comment}
\usepackage{ulem}
\usepackage{xcolor}

\theoremstyle{plain}
\newtheorem{theorem}{Theorem}[section]
\newtheorem{lemma}[theorem]{Lemma}

\theoremstyle{definition}

\numberwithin{equation}{section}

\makeatletter
\@namedef{subjclassname@2020}{%
  \textup{2020} Mathematics Subject Classification}
\makeatother

\DeclareMathOperator{\diag}{\text{diag}}
\DeclareMathOperator{\Tr}{\text{Tr}}

\def\be{\begin{equation}}
\def\ee{\end{equation}}

\def\R{\mathbb R}

\def\epsilon{\varepsilon}

\def\({\left(}
\def\){\right)}

\begin{document}

\title[Solutions near Infinity]
{Solutions of  the Special Lagrangian 
Equation\\
near Infinity}
\author[Han]{Qing Han}
\address{Department of Mathematics\\
University of Notre Dame\\
Notre Dame, IN 46556, USA} \email{qhan@nd.edu}
\author[Marchenko]{Ilya Marchenko}
\address{Department of Mathematics\\
University of Notre Dame\\
Notre Dame, IN 46556, USA}
\email{imarchen@alumni.nd.edu}


\begin{abstract}
Solutions to special Lagrangian equations near infinity, 
with supercritical phases or with semiconvexity on solutions, 
are known to be asymptotic to quadratic polynomials for dimension $n\ge 3$, 
with an extra logarithmic term for $n=2$. 
Via modified Kelvin transforms, we characterize remainders in the asymptotic expansions 
by a single function near the origin.
Such a function is smooth in even dimension, 
but only $C^{n-1,\alpha}$ in odd dimension $n$, for any $\alpha\in (0,1)$. 
\end{abstract}

\thanks{Q. Han acknowledges the support of NSF Grant DMS-2305038.
}

\date{\today}

\maketitle

\section{Introduction}\label{sec-Intro}

Classification of global solutions forms an important topic in differential equations. 
In this paper, we study behaviors of solutions  of the special Lagrangian equations
near infinity. 

We first review the corresponding results for the Monge-Amp\`{e}re equation. 
A classical result asserts that any classical convex solution of 
$\det(\nabla^2u)=1$ in $\R^n$
is a quadratic polynomial, due to J\"orgens ($n = 2$ \cite{Jorgens1954}), Calabi ($n \le 5$ \cite{Calabi1958}), and Pogorelov
($n \ge 2$ \cite{Pogorelov}). Caffarelli \cite{Caffarelli1995} extended this result
from classical solutions to viscosity solutions. 
Ferrer, Mart\'{i}nez, and Mil\'{a}n \cite{FeMaMi1999} studied convex classical solutions of $\det(\nabla^2u) = 1$ 
outside a bounded subset of $\R^2$. 
Caffarelli and Li \cite{CaLi2003} studied convex viscosity solutions of $\det(\nabla^2u) = 1$ 
outside a bounded subset of $\R^n$
and proved that solutions are well-approximated by quadratic polynomials at infinity. 
For 2-dimensional case, quadratic polynomials should be modified by an additional logarithmic term. 
Han and Wang \cite{HanWang} used a single function $v$
to characterize remainders in the assymptotic expansions via a modified Kelvin transform. 
Such a function is smooth in even dimension, 
but only $C^{n-1,\alpha}$ in odd dimension $n$, for any $\alpha\in (0,1)$. 
In the odd dimension case, they identified a single factor 
which provides the sole obstruction to the $C^n$ regularity 
and proved a decomposition result for the function $v$.

We now turn our attention to the special Lagrangian equations. 
The study of special Lagrangian submanifolds was initiated by 
Harvey and Lawson \cite{Harvey1982}, where they introduced 
the notion of special Lagrangian tangent planes and showed the volume-minimizing property 
for gradient graphs of solutions to the special Lagrangian equation using a calibration argument. 
The special Lagragian equation has the following form: 
\begin{equation}\label{eq-slag-dim-n}\sum_{i=1}^n\arctan\left(\lambda_i(\nabla^2u)\right)=\theta,\end{equation}
where $\theta$ is a constant, referred to as the \lq\lq phase\rq\rq, and 
$\lambda_i(\nabla^2u)$ denote the eigenvalues of the Hessian matrix $\nabla^2u$. 

Yuan \cite{Yuan2002, Yuan2006} studied entire solutions to \eqref{eq-slag-dim-n} and established that 
any global solutions with supercritical phases or with semiconvexity are quadratic polynomials. 
The supercritical phase refers to the requirement $|\theta|>(n-2)\pi/2$. 
This is a necessary condition, as demonstrated 
by an entire solution $u(x)= (x_1^2+x_2^2-1)e^{-x_3}+ e^{x_3} /4$ to \eqref{eq-slag-dim-n} with $\theta =\pi/2$ 
in $\mathbb R^3$ due to Warren [20].



Li, Li, and Yuan \cite{LLY2020} studied solutions to \eqref{eq-slag-dim-n} in exterior domains and 
established asymptotic expansions for solutions near infinity. 
Specifically, let $u$ be a smooth solution of \eqref{eq-slag-dim-n} in $\R^n\setminus \overline{B}_1$
with the supercritical phase or with semiconvexity. 
They proved that 

$\mathrm{(i)}$ for $n \geq 3$, there exist some $n\times n$ symmetric matrix $A$, some vector $b \in \mathbb{R}^n$, 
and some constant $c\in\mathbb{R}$,  such that  
\begin{equation}\label{eq-slag-asymptotic-dim-large-v1}
u(x)=\frac12x^TAx+b\cdot x+c+O(|x|^{2-n}) \quad\text{as } x\to\infty;
\end{equation}

$\mathrm{(ii)}$ for $n = 2$, there exist some $2\times 2$ symmetric matrix $A$, some vector $b \in \mathbb{R}^2$, 
and some constant $c, d\in\mathbb{R}$ such that
\begin{equation}\label{eq-slag-asymptotic-dim-2-v1}
u(x)=\frac12x^TAx+b\cdot x+\frac12d\log\big(x^T(I+A^2)x\big)+c+O(|x|^{-1})\quad\text{as } x\to\infty.
\end{equation}
In fact, they also derived similar estimates for derivatives of $u$. 

In this paper, we will characterize the error terms in \eqref{eq-slag-asymptotic-dim-large-v1} 
and \eqref{eq-slag-asymptotic-dim-2-v1} more precisely. 
For a positive definite matrix $B$, let $\sqrt{B}$ be a positive square root of $B$; 
namely, $\sqrt{B}$ is a positive definite matrix such that $(\sqrt{B})^2=B$. 
We will prove the following result. 

\begin{theorem}\label{thrm-slag-expansion} 
Let $u$ be a smooth solution of \eqref{eq-slag-dim-n} in $\R^n\setminus \overline{B}_1$, 
satisfying either $|\theta|>(n-2)\pi/2$ or $\nabla^2u \ge -KI$ for $n\leq4$ and 
$\nabla^2u\ge -(\frac{1}{\sqrt{3}}+\epsilon(n))I$ for $n\geq5$, for some arbitrary large constant $K$ and small 
dimensional constant $\epsilon(n)$. 
Then, 

$\mathrm{(i)}$ for $n \geq 3$, there exist some $n\times n$ symmetric matrix $A$, some vector $b \in \mathbb{R}^n$, 
some constant $c\in\mathbb{R}$,  
and a function $v$ in $B_1$, such that, for any  $x\in  \R^n\setminus \overline{B}_R$ for some $R>0$, 
\begin{equation}\label{eq-slag-expansion-identity-dim-large}u(x)=\frac{1}{2}x^TAx+b\cdot x+c+
|\sqrt{I+A^2}x|^{2-n}v\left(\frac{\sqrt{I+A^2}x}{|\sqrt{I+A^2}x|^2}\right);\end{equation}
 
$\mathrm{(ii)}$ for $n = 2$, there exist some $2\times 2$ symmetric matrix $A$, some vector $b \in \mathbb{R}^2$, 
some constant $d\in\mathbb{R}$,  
and a function $v$ in $B_1$, such that, for any $x\in  \R^2\setminus \overline{B}_R$  for some $R>0$, 
\begin{equation}\label{eq-slag-expansion-identity-dim-2}u(x)=\frac{1}{2}x^TAx+b\cdot x+\frac{1}2d\log\big(x^T(I+A^2)x\big)+
v\left(\frac{\sqrt{I+A^2}x}{|\sqrt{I+A^2}x|^2}\right).\end{equation}
Moreover, $v\in C^\infty(B_1)$ for $n$ even and $v\in C^{n-1,\alpha}(B_1)$ for $n$ odd, for any $\alpha\in (0,1)$.
\end{theorem}

In fact, the function $v$ here is defined so that \eqref{eq-slag-expansion-identity-dim-large} 
or \eqref{eq-slag-expansion-identity-dim-2} holds. 
Then, 
$v$ is a smooth function in $B_1\setminus\{0\}$. 
The main contribution of Theorem \ref{thrm-slag-expansion} is the assertion 
that $v$ can be extended to a function in $B_1$ with higher regularity. 
In particular, $v$ has no singularity at the origin. 
For $n=2$, by setting  $c=v(0)$ in \eqref{eq-slag-expansion-identity-dim-2},  the rest of 
$v\left(\frac{\sqrt{I+A^2}x}{|\sqrt{I+A^2}x|^2}\right)$ starts with the order $|x|^{-1}$. 
This is consistent with \eqref{eq-slag-asymptotic-dim-2-v1}. 
In the even dimension, $v$ is smooth. 
When the dimension $n$ is odd, $v=v(y)$ involves the expression $|y|^{n-2}y_iy_j$, 
which prevents $v$ from being better than $C^n$. 
We will prove the following decomposition. 

\begin{theorem}\label{new-slag-odd-dim} 
Let $n\ge 3$ be odd and $v$ be the function as in Theorem \ref{thrm-slag-expansion}. 
Then, there exist smooth functions $v_1$, $v_2\in C^\infty(B_1)$, with $v_2(0)=0$ and $\nabla v_2(0)=0$, such that  
\begin{equation}\label{eq-expansion-slag-odd-dim-intro} 
v(y)=v_1(y)+|y|^{n-2}v_2(y).\end{equation}
\end{theorem} 

Hence, by expanding $v$ or $v_1, v_2$ near the origin in the even or odd dimension, 
we can expand $u$ and its derivatives up to arbitrary orders near infinity. 

We briefly describe the proof of 
Theorems \ref{thrm-slag-expansion}-\ref{new-slag-odd-dim}. 
Via a modified Kelvin transform, $v$ is introduced as in \eqref{eq-slag-expansion-identity-dim-large}
or \eqref{eq-slag-expansion-identity-dim-2} as a function in $B_1\setminus\{0\}$. 
The equation \eqref{eq-slag-dim-n} for $u$ in $\R^n\setminus \overline{B}_1$ yields 
a fully nonlinear elliptic equation for $v$ in $B_1\setminus\{0\}$. 
The key step is to arrange this equation in the form 
\begin{equation}\label{eq:desired_form-intro}\Delta v + f(v;y) = 0\quad\text{in }B_1\setminus\{0\}.\end{equation}
Here, $\Delta v$ is the single linear term in $v$ and $f$ contains all higher order terms of $v$ and its derivatives up to order 2. 
We view \eqref{eq:desired_form-intro} as a perturbation of the Laplace equation
and prove that $v$ can be extended to a higher regularity function in $B_1$. 
In the even dimension, $v$ is smooth in $B_1$. 
In the odd dimension as in Theorem \ref{new-slag-odd-dim}, 
a singularity of type $|y|^{n-2}y_iy_j$ cannot be removed. 
A pointwise Schauder estimate established by the first author in 
\cite{Han2000} plays an essential role in the proof of Theorem \ref{new-slag-odd-dim}.
The process is similar to what is used in \cite{HanWang} for the study of the Monge-Amp\`{e}re equation. 
However, there are some major differences. 

Consider the function $v(y)$ defined through a modified Kelvin transform given by 
$$y=\frac{Rx}{|Rx|^2},$$
for a positive definite matrix $R$. 
Let $A$ be the symmetric matrix from the coefficients of the quadratic part in the expansion of $u$. 
For the Monge-Amp\`{e}re equation, $A$ is positive definite and we take 
$R=\sqrt{A}$. 
In the corresponding equation \eqref{eq:desired_form-intro} for the Monge-Amp\`{e}re equation, 
the only nonsmooth factor is given by $|y|^{n-2}$ for odd $n$. 
As a consequence, $|y|^{n}$ is the only factor preventing $v$ from being $C^n$. 
For the special Lagrangian equation,
the process leading to the correct form of the equation for $v$ is more complicated. 
In this case, we take $R=\sqrt{I+A^2}.$
In the equation \eqref{eq:desired_form-intro}, 
nonsmooth factors are given by $|y|^{n-4}y_iy_j$ for odd $n$. 
Hence, all factors $|y|^{n-2}y_iy_j$ prevent $v$ from being $C^n$. 

There are other important fully nonlinear elliptic equations 
in addition to the Monge-Amp\`{e}re equation and the special Lagrangian equation. 
Flanders \cite{Flanders1960} and Chang and Yuan \cite{Chang&Yuan2010} 
classified global convex solutions to inverse harmonic Hessian equations and quadratic Hessian equations, 
respectively. 
Arguments by Li, Li, and Yuan \cite{LLY2020} also lead to quadratic asymptotics for convex solutions to 
these equations in exterior domains. 

The Monge-Amp\`{e}re equation and the special Lagrangian equation 
can be included in a family of fully nonlinear elliptic equations as recorded by Warren \cite{Warren2010}. 
Bao and Liu \cite{Liu2022} obtained quadratic asymptotics for solutions to 
these equations in exterior domains under appropriate assumptions. 
In this paper, we will also study these equations and prove analogous 
Theorems \ref{thrm-slag-expansion}-\ref{new-slag-odd-dim} for these equations.

The paper is organized as follows. 
In Section \ref{sec-slag-dim3}, we study the special Lagrangian equation in dimension 3. 
We provide all necessary computations leading to the correct form of the equation of $v$. 
In Section \ref{sec-Combinatorial-Lemmas}, we establish several combinatorial lemmas. 
In Section \ref{sec-slag-dim-n}, we study the special Lagrangian equation in arbitrary dimension 
and prove Theorem \ref{thrm-slag-expansion} and Theorem \ref{new-slag-odd-dim}. 
In Section \ref{sec-slag-type}, we study a family of special Lagrangian type equations 
and prove similar results. 

\section{Special Lagrangian Equations for $n=3$}\label{sec-slag-dim3} 

In this section, we study the special Lagrangian equation for $n=3$ 
and provide a detailed computation leading to the equation \eqref{eq:desired_form-intro}. 
It is well-known that the special Lagrangian equation \eqref{eq-slag-dim-n} for $n=3$ takes the algebraic form
\begin{equation}\label{eq:4-dim3}\cos(\theta)\big[\sigma_{1}(\nabla^2u)-\sigma_{3}(\nabla^2u)\big] 
- \sin(\theta)\big[1-\sigma_{2}(\nabla^2u)\big] = 0.\end{equation}
Let $u$ be a smooth solution of \eqref{eq:4-dim3} in $\mathbb R^3\setminus\bar B_1$. 
Under the assumption of the supercritical phase or the semiconvexity as in Theorem \ref{thrm-slag-expansion}, 
Li, Li, and Yuan \cite{LLY2020} proved that 
that there exist a $3\times 3$ symmetric matrix $A$, a vector $b\in\mathbb R^3$, and a constant $c\in\mathbb R$ such that, 
for any $x\in\mathbb R^3\setminus\bar B_1$ and any integer $\ell\ge 0$, 
\begin{equation}\label{eq:special_lagrangian_sol-dim3}
u(x) = \frac{1}{2}x^TAx+b\cdot x + c + O_\ell(|x|^{-1})\quad\text{as } x\to\infty.\end{equation}
Here, the notation $O_\ell(|x|^{-1})$ means that
\begin{equation}\label{eq:O_ell-dim3}
\limsup_{x\to\infty}|x|^{-1+\ell}\Big|\nabla^\ell\Big[u(x)-\frac{1}{2}x^TAx+b\cdot x + c\Big]\Big|<\infty.\end{equation}
By \eqref{eq:special_lagrangian_sol-dim3}, it follows that $A$ also satisfies
\[\cos(\theta)\big[\sigma_{1}(A)-\sigma_{3}(A)\big] 
- \sin(\theta)\big[1-\sigma_{2}(A)\big]  = 0.\]
A simple algebraic manipulation yields 
\begin{equation}\label{eq:special_lagrangian_no_theta-dim3}\begin{split}
\big[1-\sigma_{2}(A)\big]\big[\sigma_{1}(\nabla^2u)-\sigma_{3}(\nabla^2u)\big] 
- \big[\sigma_{1}(A)-\sigma_{3}(A)\big]\big[1-\sigma_{2}(\nabla^2u)\big] = 0.\end{split}\end{equation}

We assume without a loss of generality that the matrix $A$ is diagonal, given by 
$$A=\diag(\lambda_1, \lambda_2, \lambda_3).$$ 
In the most general case, $A$ can be any symmetric matrix, 
but we can always diagonalize it by an appropriate rotation. 
For a $3\times 3$ diagonal matrix $R$ to be determined, 
set 
\[y=\frac{R x}{|R x|^2}.\]
We also write $A=(A_{ij})$ and $R=(R_{ij})$ for convenience. 
In view of \eqref{eq:special_lagrangian_sol-dim3}, we introduce a function $v$ in $B_1\setminus\{0\}$ such that 
\begin{equation}\label{eq:special_lagrangian_kelvin-dim3}
u(x) = \frac{1}{2}x^TAx+b\cdot x + c + |y|^{-1}v(y).\end{equation}
It is obvious $v\in C^\infty(B_1\setminus\{0\})$. Our goal is to analyze the regularity of $v$ at $y=0$. 
We will demonstrate that $v\in C^{2,\alpha}(B_1)$ for any $\alpha\in (0,1)$. 
We will also show that factors of the type $|y|y_iy_j$ provide the only obstruction to the higher regularity.

We can verify by a direct computation that
\begin{equation}\label{eq:uij-dim3}\nabla^2u = A + |y|^3N(v),\end{equation}
where $N(v)=(N_{ij}(v))$ is the matrix given by 
\begin{equation}\label{eq:mij_tilde-dim3}\begin{split}
N_{ij}(v) &= \big[3v+12\langle y,\nabla v\rangle+4y_ky_m v_{km}\big]R_{ii}R_{jj}y_iy_j|y|^{-2} 
\\&\qquad-\big[v+2\langle y,\nabla v\rangle\big]R_{ii}R_{jj}\delta_{ij}- 3R_{ii}R_{jj}(y_iv_j+y_jv_i)
\\&\qquad-2R_{ii}R_{jj}(y_iy_kv_{kj} + y_jy_m v_{im})+R_{ii}R_{jj}|y|^2v_{ij}.\end{split}\end{equation}
Here, the derivatives of $v$ are taken with respect to $y$. 
We should point out that $i$ and $j$ are fixed in \eqref{eq:mij_tilde-dim3}.
It will be useful to write $N(v)$ as 
\begin{equation}\label{eq:mij_tilde-v2-dim3}
N_{ij}(v) = R_{ii}R_{jj}M_{ij}(v),\end{equation}
where 
\begin{equation}\label{eq:mij-dim3}\begin{split}
M_{ij}(v) &= \big[3v+12\langle y,\nabla v\rangle+4y_ky_m v_{km}\big]y_iy_j|y|^{-2} \\
&\qquad-\big[v+2\langle y,\nabla v\rangle\big]\delta_{ij}- 3(y_iv_j+y_jv_i)\\
&\qquad-2(y_iy_kv_{kj} + y_jy_m v_{im})+|y|^2v_{ij}.\end{split}\end{equation}
We omit the dependence on $v$ from our notations for the matrices $M,N$ from now on.

By \eqref{eq:O_ell-dim3}, we have, for any integer $\ell\geq0$, 
\[\sup_{B_1\setminus\{0\}}|y|^\ell |\nabla^\ell v(y)|<\infty,\]
and in particular, $v$ is bounded in $B_1\setminus\{0\}$. 
Furthermore, in each term of $N_{ij}$, every $v_i$-term is coupled with a factor $y_j$ 
and every $v_{ij}$-term is coupled with at least a factor $y_ky_m$ 
and so $N_{ij}$ is a bounded function even though the factor $|y|^{-2}y_iy_j$ is not continuous at $y=0$.

By \eqref{eq:uij-dim3}, we have 
\begin{align*}
\sigma_1(\nabla^2u)&=\sigma_1(A)+|y|^3\sigma_1(N),\\
\sigma_2(\nabla^2u)&=\sigma_2(A)+|y|^3\big[(\lambda_1+\lambda_2)N_{33}+(\lambda_2+\lambda_3)N_{11}
+(\lambda_3+\lambda_1)N_{22}\big]
+|y|^6\sigma_2(N),
\end{align*}
and 
\begin{align*}
\sigma_3(\nabla^2u)&=\sigma_3(A)+|y|^3\big[\lambda_1\lambda_2N_{33}+\lambda_2\lambda_3N_{11}
+\lambda_3\lambda_1N_{22}\big]\\
&\qquad +|y|^6\big[\lambda_1(N_{22}N_{33}-N^2_{23})+\lambda_2(N_{11}N_{33}-N^2_{13})
+\lambda_3(N_{11}N_{22}-N^2_{12})\big]\\
&\qquad +|y|^9\sigma_3(N).
\end{align*}
By substituting these expressions in \eqref{eq:special_lagrangian_no_theta-dim3}, 
we rearrange \eqref{eq:special_lagrangian_no_theta-dim3} according to the power of $|y|$. 
We first note that terms not involving $|y|$ cancel each other. All remaining terms involve 
$|y|^3$, $|y|^6$, or $|y|^9$. By eliminating a factor of $|y|^3$, we conclude 
\begin{equation}\label{eq:n3_explicit-dim3}
J_1+|y|^3J_2+|y|^6J_3=0,\end{equation}
where 
\begin{align*}
J_1&=\big[1-\sigma_{2}(A)\big]\big[\sigma_{1}(N)-(\lambda_1\lambda_2N_{33}+\lambda_2\lambda_3N_{11}
+\lambda_3\lambda_1N_{22})\big] \\
&\qquad+ \big[\sigma_{1}(A)-\sigma_{3}(A)\big]
\big[(\lambda_1+\lambda_2)N_{33}+(\lambda_2+\lambda_3)N_{11}
+(\lambda_3+\lambda_1)N_{22}\big],\end{align*}
and 
\begin{align*}
J_2&=-\big[1-\sigma_{2}(A)\big]\big[\lambda_1(N_{22}N_{33}-N^2_{23})+\lambda_2(N_{11}N_{33}-N^2_{13})\\
&\qquad\qquad\qquad+\lambda_3(N_{11}N_{22}-N^2_{12})\big] 
+ \big[\sigma_{1}(A)-\sigma_{3}(A)\big]\sigma_2(N),\\
J_3&=-\big[1-\sigma_{2}(A)\big]\sigma_3(N).\end{align*}
We point out that $N$ involves discontinuous factors $y_iy_j/|y|^2$ according to \eqref{eq:mij_tilde-dim3}. 
We next expand $J_1$, $J_2$, and $J_3$ in terms of $M_{ij}$ and choose the matrix $R$ in this process. 

We first examine $J_1$. A simple rearrangement according to $N_{ii}$ leads to  
\begin{align*}
J_1&=\big[\big(1-\sigma_{2}(A)\big)(1-\lambda_2\lambda_3\big) 
+ \big(\sigma_{1}(A)-\sigma_{3}(A)\big)
(\lambda_2+\lambda_3)\big]N_{11}\\
&\qquad +\big[\big(1-\sigma_{2}(A)\big)(1-\lambda_3\lambda_1\big) 
+ \big(\sigma_{1}(A)-\sigma_{3}(A)\big)
(\lambda_3+\lambda_1)\big]N_{22}\\
&\qquad+\big[\big(1-\sigma_{2}(A)\big)(1-\lambda_1\lambda_2\big) 
+ \big(\sigma_{1}(A)-\sigma_{3}(A)\big)
(\lambda_1+\lambda_2)\big]N_{33}.\end{align*}
A straightforward computation yields 
\begin{align*}
J_1=(1+\lambda_2^2)(1+\lambda_3^3)N_{11}
+(1+\lambda_3^2)(1+\lambda_1^3)N_{22}
+(1+\lambda_1^2)(1+\lambda_2^3)N_{33}.\end{align*}
By \eqref{eq:mij_tilde-v2-dim3}, we get 
\begin{align*}
J_1=(1+\lambda_2^2)(1+\lambda_3^3)R_{11}^2M_{11}
+(1+\lambda_3^2)(1+\lambda_1^3)R_{22}^2M_{22}
+(1+\lambda_1^2)(1+\lambda_2^3)R^2_{33}M_{33}.\end{align*}
Now, we take 
\begin{equation}\label{eq-expression-R-dim3}
R=\diag\big((1+\lambda_1^2)^{1/2}, (1+\lambda_2^2)^{1/2}, (1+\lambda_3^2)^{1/2}\big).\end{equation}
Then, 
\begin{align*}
J_1=(1+\lambda_1^3)(1+\lambda_2^2)(1+\lambda_3^3)\sigma_1(M).\end{align*}
A straightforward computation based on \eqref{eq:mij-dim3} yields 
$$\sigma_1(M)=|y|^2\Delta v.$$ 
Therefore, we obtain 
\begin{equation}\label{eq:linear-dim3}
J_1=|y|^2(1+\lambda_1^3)(1+\lambda_2^2)(1+\lambda_3^3)\Delta v.\end{equation}

We now examine $J_2$. First, note that 
$$\sigma_2(N)=(N_{22}N_{33}-N^2_{23})+(N_{11}N_{33}-N^2_{13})\\
+(N_{11}N_{22}-N^2_{12}).$$ 
We can write 
\begin{align*}
J_2&=\big[\sigma_{1}(A)-\sigma_{3}(A)-\lambda_1\big(1-\sigma_{2}(A)\big)\big](N_{22}N_{33}-N^2_{23})\\
&\qquad +\big[\sigma_{1}(A)-\sigma_{3}(A)-\lambda_2\big(1-\sigma_{2}(A)\big)\big](N_{33}N_{11}-N^2_{13})\\
&\qquad+\big[\sigma_{1}(A)-\sigma_{3}(A)-\lambda_3\big(1-\sigma_{2}(A)\big)\big](N_{11}N_{22}-N^2_{12}).
\end{align*}
A straightforward calculation yields 
\begin{align*}
J_2&=(1+\lambda_1^2)(\lambda_2+\lambda_3)(N_{22}N_{33}-N^2_{23})\\
&\qquad +(1+\lambda_2^2)(\lambda_3+\lambda_1)(N_{33}N_{11}-N^2_{13})\\
&\qquad+(1+\lambda_3^2)(\lambda_1+\lambda_2)(N_{11}N_{22}-N^2_{12}).
\end{align*}
By \eqref{eq:mij_tilde-v2-dim3} and \eqref{eq-expression-R-dim3}, we have, for any $1\le i, j\le 3$ with $i\neq j$, 
$$N_{ii}N_{jj}-N_{ij}^2=R_{ii}^2R^2_{jj}(M_{ii}M_{jj}-M_{ij}^2)=(1+\lambda_i^2)(1+\lambda_j^2)(M_{ii}M_{jj}-M_{ij}^2).$$
Hence, 
\begin{equation}\label{eq:quad-dim3}\begin{split}
J_2&=(1+\lambda_1^2)(1+\lambda_2^2)(1+\lambda_3^2)
\big[(\lambda_2+\lambda_3)(M_{22}M_{33}-M^2_{23})\\
&\qquad +(\lambda_3+\lambda_1)(M_{33}M_{11}-M^2_{13})
+(\lambda_1+\lambda_2)(M_{11}M_{22}-M^2_{12})\big].
\end{split}\end{equation}

Last, by \eqref{eq:mij_tilde-v2-dim3} and \eqref{eq-expression-R-dim3} again, we have 
\begin{equation}\label{eq:cub-dim3}\begin{split}
\sigma_3(N)&=-\big[1-\sigma_{2}(A)\big]R_{11}^2R_{22}^2R^3_{33}\sigma_3(M)\\
&=-\big[1-\sigma_{2}(A)\big](1+\lambda_1^2)(1+\lambda_2^2)(1+\lambda_3^2)\sigma_3(M).
\end{split}\end{equation}
We note that the expressions in the right-hand side of \eqref{eq:linear-dim3}, \eqref{eq:quad-dim3}, and \eqref{eq:cub-dim3}
have a common factor $(1+\lambda_1^2)(1+\lambda_2^2)(1+\lambda_3^2)$. 
We emphasize this is due to the choice of the matrix $R$ in \eqref{eq-expression-R-dim3}. 
We now substitute these expressions in \eqref{eq:n3_explicit-dim3} and 
eliminate this common factor and an additional factor $|y|^2$. We obtain 
\begin{equation}\label{eq:n3_explicit-dim3-v2}
\Delta v+|y|I_2+|y|^4I_3=0,\end{equation}
where 
\begin{equation}\label{eq:quad-dim3-v2}\begin{split}
I_2&=(\lambda_2+\lambda_3)(M_{22}M_{33}-M^2_{23})\\
&\qquad +(\lambda_3+\lambda_1)(M_{33}M_{11}-M^2_{13})
+(\lambda_1+\lambda_2)(M_{11}M_{22}-M^2_{12}),
\end{split}\end{equation}
and 
\begin{equation}\label{eq:cub-dim3-v2}\begin{split}
I_3&=-\big[1-\sigma_{2}(A)\big]\sigma_3(M).
\end{split}\end{equation}
We point out that each $M_{ij}$ as in \eqref{eq:mij-dim3} is a polynomial of $v$ and its derivatives up to order 2. 
In \eqref{eq:n3_explicit-dim3-v2}, $\Delta v$ is the single linear term in $v$, 
$I_2$ involves all terms quadratic in $v$ and its derivatives up to order 2, 
and $I_3$ involves all terms cubic in $v$ and its derivatives up to order 2. 
Hence, \eqref{eq:n3_explicit-dim3-v2} is a fully nonlinear equation of $v$ of order 2. 
Later on, we will treat the equation \eqref{eq:n3_explicit-dim3-v2} as a perturbation of the Laplace equation near the origin. 
The function $|y|$ is Lipschitz in $B_1$, but not $C^1$. 
This suggests that the best regularity for $v$ is $C^{2,\alpha}$ in $B_1$, for any $\alpha\in (0,1)$. 
However, we may have much worse regularity since $M_{ij}$ involves discontinuous factors $y_iy_j/|y|^2$. 

We next expand $M_{ii}M_{jj}-M^2_{ij}$ for $1\le i\neq j\le 3$ and $\sigma_3(M)$ in terms of $v$. 
In view of \eqref{eq:mij-dim3}, we write 
\begin{equation*}
M_{ij} =K_{ij}+ \frac{y_iy_j}{|y|^{2}}L,\end{equation*}
where 
\begin{align}\label{eq-expressions-K-ij}K_{ij}=-\big[v+2\langle y,\nabla v\rangle\big]\delta_{ij}- 3(y_iv_j+y_jv_i)
-2(y_iy_kv_{kj} + y_jy_m v_{im})+|y|^2v_{ij},\end{align}
and 
\begin{align}\label{eq-expressions-L}L=3v+12\langle y,\nabla v\rangle+4y_ky_m v_{km}.\end{align}
We point out that $K_{ij}$ and $L$ are linear in $v$ and its derivatives up to order 2. 
A crucial observation here is that the $3\times 3$ matrix $(Ly_iy_j/|y|^2)$ is of rank-one. 
As a consequence, in the expressions of the determinants of $M$ and any of its $2\times 2$ principal minors, 
there are no higher orders of $Ly_iy_j/|y|^2$. 
For any $1\le i\neq j\le 3$, a straightforward computation yields 
$$M_{ii}M_{jj}-M^2_{ij}=(K_{ii}K_{jj}-K^2_{ij})+\frac{L}{|y|^2}(y_i^2K_{jj}+y_j^2K_{ii}-2y_iy_jK_{ij}).$$
Then, 
\begin{align*}
|y|^2I_2&=(\lambda_1+\lambda_2)\big[|y|^2(K_{11}K_{22}-K_{12}^2)+L(y_1^2K_{22}+y_2^2K_{11}-2y_1y_2K_{12})\big]\\
&\qquad+(\lambda_2+\lambda_3)\big[|y|^2(K_{22}K_{33}-K_{23}^2)+L(y_2^2K_{33}+y_3^2K_{22}-2y_2y_3K_{23})\big]\\
&\qquad+(\lambda_3+\lambda_1)\big[|y|^2(K_{33}K_{11}-K_{31}^2)+L(y_3^2K_{11}+y_1^2K_{33}-2y_3y_1K_{31})\big].
\end{align*}
We note that $I_2$ is coupled with a factor $|y|$ in \eqref{eq:n3_explicit-dim3-v2}. 
We need to divide a factor $|y|$ when we substitute $|y|I_2$ in \eqref{eq:n3_explicit-dim3-v2}. 
This creates a singular factor $|y|^{-1}$. 
Next, a similar computation yields 
\begin{align*}
|y|^2\sigma_3(M)&=y_1^2L(K_{22}K_{33}-K_{23}^2)+y_2^2L(K_{11}K_{33}-K_{13}^2)
+y_3^2L(K_{11}K_{22}-K_{12}^2)\\
&\qquad\qquad+2y_1y_2L(K_{12}K_{23}-K_{12}K_{33})+2y_2y_3L(K_{21}K_{31}-K_{23}K_{11})\\
&\qquad\qquad+2y_2y_3L(K_{21}K_{31}-K_{23}K_{11})+\sigma_3(K). 
\end{align*}
We note that $I_3=-\big[1-\sigma_{2}(A)\big]\sigma_3(M)$ is coupled with a factor $|y|^4$ in \eqref{eq:n3_explicit-dim3-v2}. 
We can use a factor $|y|^2$ from $|y|^4$ to form $|y|^2\sigma_3(M)$.

In conclusion, we can write \eqref{eq:n3_explicit-dim3-v2} as 
\begin{equation}\label{eq:14-dim3}
\Delta v + |y|^{-1}\sum_{i,j=1}^3y_iy_jF_{ij}+|y|^2\sum_{i,j=1}^3y_iy_jG_{ij}=0 \quad\text{in }B_1\setminus\{0\},\end{equation}
where $F_{ij}$ and $G_{ij}$ are polynomials in $y_i$, $v$, $\nabla v$, and $\nabla^2v$. 
In fact, $F_{ij}$ are homogeneous quadratic in $v$, $\nabla v$, and $\nabla^2v$, 
and $G_{ij}$ are homogeneous cubic in $v$, $\nabla v$, and $\nabla^2v$.  
The coefficient $y_iy_j|y|^{-1}$ for $F_{ij}$ has homogeneous degree 1 and is Lipschitz in $B_1$. 
Then, we can extend $v$ to $y=0$ and that the extended $v$ is $C^{2,\alpha}$ in $B_1$ for any $\alpha\in(0,1)$.
This is based on a removable singularity result established in \cite{HanWang}. 
Details will be given in Section \ref{sec-slag-dim-n}. 

We next demonstrate that $v$ is actually $C^{2,1}$ near $y=0$. 
We need to identify the homogeneous degree 3 part of $v$ which is not $C^3$ at $y=0$. 
To this end, we compute $F_{ij}$ at $y=0$. 
Since $v$ is $C^{2,\alpha}$ in $B_1$, then $v(0)$, $\nabla v(0)$, and $\nabla^2v(0)$ exists. 
In \eqref{eq-expressions-K-ij} and \eqref{eq-expressions-L}, each $v_i$ is coupled with some $y_j$, 
and each $v_{ij}$ is coupled with some $y_ky_l$. Hence, 
\begin{align*}K_{ij}=-v(0)\delta_{ij}+O(|y|), \quad L=3v(0)+O(|y|).\end{align*}
By a simple substitution, we get 
\begin{align*}
|y|^2I_2&=v^2(0)\big[(\lambda_1+\lambda_2)\big(|y|^2+3(y_1^2+y_2^2)\big)
+(\lambda_2+\lambda_3)\big(|y|^2+3(y_2^2+y_3^2)\big)\\
&\qquad+(\lambda_3+\lambda_1)\big(|y|^2+3(y_3^2+y_1^2)\big)\big]+O(|y|^3).
\end{align*}
Note that 
\begin{align*}
&(\lambda_1+\lambda_2)(y_1^2+y_2^2)
+(\lambda_2+\lambda_3)(y_2^2+y_3^2)
+(\lambda_3+\lambda_1)(y_3^2+y_1^2)\\
&\qquad=\sigma_1(A)|y|^2+\lambda_1y_1^2+\lambda_2y_2^2+\lambda_3y_3^2.
\end{align*}
Hence, 
\begin{align*}
|y|^2I_2=v^2(0)\big[5\sigma_1(A)|y|^2+3(\lambda_1y_1^2+\lambda_2y_2^2+\lambda_3y_3^2)\big]+O(|y|^3).
\end{align*}
Now, the equation \eqref{eq:14-dim3} reduces to 
\begin{align*}
\Delta v +v^2(0) |y|^{-1}\big[5\sigma_1(A)|y|^2+3(\lambda_1y_1^2+\lambda_2y_2^2+\lambda_3y_3^2)\big]+O(|y|^2)=0 
\quad\text{in }B_1\setminus\{0\}.\end{align*}
We need to point out that $O(|y|^2)$ contains the homogeneous degree 2 factors like $|y|y_1$, which is not $C^2$. 
Set 
$$\overline{Q}_2=v^2(0) \big[5\sigma_1(A)|y|^2+3(\lambda_1y_1^2+\lambda_2y_2^2+\lambda_3y_3^2)\big].$$
This is a homogeneous quadratic polynomial in $\mathbb R^3$. 
Then, $v$ satisfies 
\begin{align*}
\Delta v +|y|^{-1}\overline{Q}_2+O(|y|^2)=0 
\quad\text{in }B_1\setminus\{0\}.\end{align*}
We now claim that there exists a unique homogeneous quadratic polynomial $Q_2$ in $\mathbb R^3$ such that 
$$\Delta(|y|Q_2)=|y|^{-1}\overline{Q}_2.$$
We postpone the proof of the uniqueness to Section \ref{sec-slag-dim-n}. For the existence, we can write such $Q_2$ explicitly by 
\begin{align*}
{Q}_2=v^2(0)\Big[\frac13\sigma_1(A)|y|^2+\frac12(\lambda_1y_1^2+\lambda_2y_2^2+\lambda_3y_3^2)\Big].\end{align*}	
Then, 
\begin{align*}
\Delta (v-|y|Q_2) +O(|y|^2)=0 
\quad\text{in }B_1.\end{align*}
We will prove $v-|y|Q_2\in C^{3,\alpha}(B_1)$. 
We should point out that $O(|y|^2)$ contains the second derivatives of $v$, 
which are only $C^\alpha$ at the moment. Proceeding inductively, we will prove that for each integer $\ell\ge 3$, 
there is a polynomial $q_{\ell-1}$, with $q_{\ell-1}(0)=0$ and $\nabla q_{\ell-1}(0)=0$, such that, for any $\alpha\in (0,1)$,  
$$v-|y|q_{\ell-1}\in C^{\ell,\alpha}(B_1).$$
By another induction process, we will prove that, for any $y\in B_1$, 
$$v(y)=v_1(y)+|y|v_2(y),$$ 
for some $v_1, v_2\in C^\infty(B_1)$ with $v_2(0)=0$ and $\nabla v_2(0)=0$. 
Details will be provided for general $n$ in Sections \ref{sec-slag-dim-n}.

\section{Combinatorial Lemmas}\label{sec-Combinatorial-Lemmas}

In this section, we introduce some notations and prove several combinatorial lemmas which will be useful to us later. 
Lemma \ref{lm:sigma_k_to_sigma_hat} and Lemma \ref{lm:special_lagrangian_linear}
will be used in the proof of Theorem \ref{thrm-slag-expansion} 
and generalize the corresponding results we showed in Section \ref{sec-slag-dim3} to higher dimensions. 
Lemma \ref{lm:sigma_bar_to_sigma_k} and Lemma \ref{lm:linear_tau_pi/4_pi/2} will be needed in 
the proof of Theorem \ref{thm:expansions_slag_type}, in which 
we show that an analog of Theorem \ref{thrm-slag-expansion} 
holds for not just the Monge-Amp\`ere and special Lagrangian equations, 
but also for an entire family of fully nonlinear elliptic equations.

Throughout the following discussion, we assume 
\begin{align}\label{eq-expression-A}A=\diag(\lambda_1,\ldots,\lambda_n),\end{align}
unless specified otherwise.
We denote by $\hat \sigma_{k,i}(A)$ the terms of $\sigma_{k}(A)$ 
which do not include $\lambda_i$. 
For example, for $n=3$, 
$\sigma_2(A) = \lambda_1\lambda_2+\lambda_2\lambda_3+\lambda_1\lambda_3$ 
and $\hat\sigma_{2,1}(A) = \lambda_2\lambda_3$. 
We set $\hat\sigma_{0,i}(A)=1$ and  $\hat\sigma_{k,i}(A)=0$ for $k<0$, 
exactly as in the definition of the $\sigma_k(A)$ for $k\leq0$. 
We also introduce two associated matrices
\begin{equation}\label{eq:A_ilower}A_{(i)}=\diag(\lambda_1,\ldots,\lambda_{i-1},0,\lambda_{i+1},\ldots,\lambda_n),\end{equation}
and
\begin{equation}\label{eq:A_iupper}A^{(i)}=\diag(\lambda_1,\ldots,\lambda_{i-1},1,\lambda_{i+1},\ldots,\lambda_n).\end{equation}
Then, we have
\begin{equation}\label{eq:hat_raise_lower}\hat\sigma_{k-1,i}(A)=\sigma_k(A^{(i)})-\sigma_k(A_{(i)}).\end{equation}

\begin{lemma}\label{lm:sigma_k_to_sigma_hat}
Let $A$ be given by \eqref{eq-expression-A}, $B$ be an $n\times n$ symmetric matrix, and $k\geq0$. 
Then, the terms of $\sigma_k(A+B)$ which are linear in $B_{ij}$ have the form 
$$\sum_{i = 1}^n\hat \sigma_{k-1,i}(A)B_{ii}.$$
\end{lemma}

\begin{proof}
Recall that we can express $\sigma_k(A+B)$ as the sum of all $k\times k$ principal minors of $A+B$. 
A $k\times k$ principal minor of $A+B$ has the form
\begin{equation}\begin{vmatrix}
\lambda_{j_1}+B_{j_1j_1}&B_{j_1j_2}&\ldots&B_{j_1j_k}\\
B_{j_2j_1}&\lambda_{j_2}+B_{j_2j_2}&\ldots&B_{j_1j_k}\\
\vdots&\vdots&\ddots&\ldots\\
B_{j_k,j_1}&B_{j_k,j_2}&\ldots&\lambda_{j_k}+B_{j_kj_k}
\end{vmatrix}\label{eq:k_k_minor}\end{equation}
for $1\le j_1<\ldots<j_k\le n$. 
Expanding \eqref{eq:k_k_minor} and writing the terms which are linear in the entries of $B$ first, 
we see that the $k\times k$ principal minor has the form
\[(\lambda_{j_2}\cdot\ldots\cdot \lambda_{j_k})B_{j_1j_1}
+\ldots+(\lambda_{j_1}\cdot\ldots\cdot \lambda_{j_{k-1}})B_{j_kj_k}+\text{higher order terms}.\]
Summing over all such minors yields the desired result.
\end{proof}

\begin{lemma}\label{lm:special_lagrangian_linear}
Let $A$ be given by \eqref{eq-expression-A} for $n\geq3$. Then, for $i=1, \cdots, n$, 
\begin{equation}\label{eq:lemma_2_result}\begin{split}
&\bigg(\sum_{0\leq 2k\leq n}(-1)^k\sigma_{2k}(A)\bigg)
\bigg(\sum_{1\leq 2k+1\leq n}(-1)^k\hat \sigma_{2k,i}(A)\bigg)\\
&\qquad\qquad-\bigg(\sum_{1\leq 2k+1\leq n}(-1)^k\sigma_{2k+1}(A)\bigg)
\bigg(\sum_{0\leq 2k\leq n}(-1)^k\hat \sigma_{2k-1,i}(A)\bigg) \\
&\qquad=(1+\lambda_1^2)\cdot\ldots\cdot(1+\lambda_{i-1}^2)(1+\lambda_{i+1}^2)
\cdot\ldots\cdot(1+\lambda_n^2).\end{split}\end{equation}
\end{lemma}

\begin{proof}
We prove this by induction on $n$. We can verify by a direct computation that the result is true for $n=3$, our base case. 
In fact, such a verification can be found in the first half of Section \ref{sec-slag-dim3}. 
Assume that the lemma holds for some $n\geq3$, and we will prove that the lemma holds for $n+1$. 

Fix $i\in\{1,\ldots,n\}$. Since we are studying the $(n+1)$-case, we have
$$A=\diag(\lambda_1,\ldots,\lambda_{n+1}).$$ 
Set
\begin{equation}\label{eq:A_hat}\hat A=\diag(\lambda_1,\ldots,\lambda_{n}).\end{equation}
We emphasize that $\hat A$ is an $n\times n$ diagonal matrix. 

For simplicity of notation, we set
\[\begin{split}
I(A)&=\bigg(\sum_{0\leq 2k\leq n+1}(-1)^k\sigma_{2k}(A)\bigg)
\bigg(\sum_{1\leq 2k+1\leq n+1}(-1)^k\hat \sigma_{2k,i}(A)\bigg)\\
&\qquad-\bigg(\sum_{1\leq 2k+1\leq n+1}(-1)^k\sigma_{2k+1}(A)\bigg)
\bigg(\sum_{0\leq 2k\leq n+1}(-1)^k\hat \sigma_{2k-1,i}(A)\bigg),\end{split}\]
and $I(\hat A)$ similarly. Since $\hat A$ is an $n\times n$ matrix, 
summations are up to $n$. In other words, we have 
\[\begin{split}
I(\hat A) &= \bigg(\sum_{0\leq 2k\leq n}(-1)^k\sigma_{2k}(\hat A)\bigg)
\bigg(\sum_{1\leq 2k+1\leq n}(-1)^k\hat \sigma_{2k,i}(\hat A)\bigg) \\
&\qquad-\bigg(\sum_{1\leq 2k+1\leq n}(-1)^k\sigma_{2k+1}(\hat A)\bigg)
\bigg(\sum_{0\leq 2k\leq n}(-1)^k\hat \sigma_{2k-1,i}(\hat A)\bigg).\end{split}\]
We will prove that
\begin{equation}\label{eq:I_equals_II}I(A)=(1+\lambda_{n+1}^2)\,I(\hat A).\end{equation}
By the induction hypothesis, we have 
\[\begin{split}
I(\hat A)=(1+\lambda_1^2)\cdot\ldots\cdot(1+\lambda_{i-1}^2)
(1+\lambda_{i+1}^2)\cdot\ldots\cdot(1+\lambda_n^2).\end{split}\]
We then obtain the desired identity.

Note that
\[\sigma_j(A)=\sigma_j(\hat A)+\lambda_{n+1}\sigma_{j-1}(\hat A),\]
for any $j$, and the same is true for $\hat\sigma_{j,i}$. 
Then, we write 
\begin{equation}\label{eq:I_equals_II_LHS}\begin{split}
I(A)&=\bigg(\sum_{0\leq 2k\leq n+1}(-1)^k\big(\sigma_{2k}(\hat A)+\lambda_{n+1}\sigma_{2k-1}(\hat A)\big)\bigg)\\
&\qquad\qquad
\cdot\bigg(\sum_{1\leq 2k+1\leq n+1}(-1)^k\big(\hat \sigma_{2k,i}(\hat A)+\lambda_{n+1}\hat \sigma_{2k-1,i}(\hat A)\big)\bigg)\\
&\qquad -\bigg(\sum_{1\leq 2k+1\leq n+1}(-1)^k\big(\sigma_{2k+1}(\hat A)+\lambda_{n+1}\sigma_{2k}(\hat A)\big)\bigg)\\
&\qquad\qquad
\cdot\bigg(\sum_{0\leq 2k\leq n+1}(-1)^k \big(\hat\sigma_{2k-1,i}(\hat A)+\lambda_{n+1}\hat\sigma_{2k-2,i}(\hat A)\big)\bigg).
\end{split}\end{equation}
We now expand \eqref{eq:I_equals_II_LHS} according to the degree of $\lambda_{n+1}$. 
First we examine the terms in the expansion of \eqref{eq:I_equals_II_LHS} from which the factor $\lambda_{n+1}$ is absent. 
These terms are
\[I_1=\bigg(\sum_{0\leq 2k\leq n+1}(-1)^k\sigma_{2k}(\hat A)\bigg)
\bigg(\sum_{1\leq 2k+1\leq n+1}(-1)^k\hat\sigma_{2k,i}(\hat A)\bigg),\]
and
\[I_2=-\bigg(\sum_{1\leq 2k+1\leq n+1}(-1)^k\sigma_{2k+1}(\hat A)\bigg)
\bigg(\sum_{0\leq 2k\leq n+1}(-1)^k\hat\sigma_{2k-1,i}(\hat A)\bigg).\]
Relabeling the summation indices and using the fact that $\sigma_{n+1}(\hat A)=0$ since $A$ is an $n\times n$ matrix, 
we see that
\[\begin{split}I_1+I_2&= \bigg(\sum_{0\leq 2k\leq n}(-1)^k\sigma_{2k}(\hat A)\bigg)
\bigg(\sum_{0\leq 2k\leq n}(-1)^k\hat\sigma_{2k,i}(\hat A)\bigg)\\
&\qquad-\bigg(\sum_{1\leq 2k+1\leq n}(-1)^k\sigma_{2k+1}(\hat A)\bigg)
\bigg(\sum_{0\leq 2k\leq n}(-1)^k\hat\sigma_{2k-1,i}(\hat A)\bigg) = I(\hat A).\end{split}\]
Next, we examine the terms from \eqref{eq:I_equals_II_LHS} which are linear in $\lambda_{n+1}$. 
These terms are given by
\[\begin{split}
I_3&=\lambda_{n+1}\bigg(\sum_{0\leq 2k\leq n+1}(-1)^k\sigma_{2k}(\hat A)\bigg)
\bigg(\sum_{1\leq 2k+1\leq n+1}(-1)^k\hat\sigma_{2k-1,i}(\hat A)\bigg),\\
I_4&=\lambda_{n+1}\bigg(\sum_{0\leq 2k\leq n+1}(-1)^k\sigma_{2k-1}(\hat A)\bigg)
\bigg(\sum_{1\leq 2k+1\leq n+1}(-1)^k\hat\sigma_{2k,i}(\hat A)\bigg),\\
I_5&=-\lambda_{n+1}\bigg(\sum_{1\leq 2k+1\leq n+1}(-1)^k\sigma_{2k+1}(\hat A)\bigg)
\bigg(\sum_{0\leq 2k\leq n+1}(-1)^k\hat\sigma_{2k-2,i}(\hat A)\bigg),\\
I_6&=-\lambda_{n+1}\bigg(\sum_{1\leq 2k+1\leq n+1}(-1)^k\sigma_{2k}(\hat A)\bigg)
\bigg(\sum_{0\leq 2k\leq n+1}(-1)^k\hat\sigma_{2k-1,i}(\hat A)\bigg).\end{split}\]
Once again, by rearranging the summation indices, we have
\[I_3+I_4+I_5+I_6=0.\]
Finally, the terms which are quadratic in $\lambda_{n+1}$ are
\[\begin{split}
I_7&=-\lambda_{n+1}^2\bigg(\sum_{1\leq 2k+1\leq n+1}(-1)^k\sigma_{2k}(\hat A)\bigg)
\bigg(\sum_{0\leq2k\leq n+1}(-1^k\hat\sigma_{2k-2,i}(\hat A)\bigg),\\
I_8&=\lambda_{n+1}^2\bigg(\sum_{0\leq 2k\leq n+1}(-1)^k\sigma_{2k-1}(\hat A)\bigg)
\bigg(\sum_{1\leq 2k+1\leq n+1}(-1)^k\hat\sigma_{2k-1,i}(\hat A)\bigg),\end{split}\]
and as before, we get
\[I_7+I_8=\lambda_{n+1}^2 I(\hat A).\]
Therefore, we obtain 
\[I(A)= I_1+I_2+I_3+I_4+I_5+I_6+I_7+I_8=(1+\lambda_{n+1}^2)\,I(\hat A),\]
which is exactly \eqref{eq:I_equals_II}. 

For $i=n+1$, the proof is essentially the same.
\end{proof}

In the rest of this section, we prove several lemmas which will be needed in 
the proof of Theorem \ref{thm:expansions_slag_type}. 
For any $a,b\in\mathbb R$ and any $n\times n$ symmetric matrix $B$, we set
\begin{equation}\label{eq:sigma_bar_def}
\bar\sigma_k(B)=\sigma_n\big(\lambda_j(B+(a+b)I)\big)\sigma_{k}\Big(\frac{\lambda_j(B+(a-b)I)}{\lambda_j(B+(a+b)I)}\Big),
\end{equation}
where $\sigma_{k}\big(\frac{\lambda_j(B+(a-b)I)}{\lambda_j(B+(a+b)I)}\big)$ 
is the elementary symmetric degree-$k$ polynomial of $\frac{\lambda_j(B+(a-b)I)}{\lambda_j(B+(a+b)I)}$. 
Note that $\bar\sigma_k(B)$ is a degree-$n$ polynomial in $\lambda_j(B)+(a\pm b)$. 
For our purposes, $a$ and $b$ will be given constants
so there will be no ambiguity about the value of the $\bar\sigma_k(B)$ 
and we omit the dependence on $a$ and $b$ from our notation. 
In the next result, we prove a relation  between $\bar\sigma_k(B)$ and $\sigma_l(B)$ for $l=1,\ldots,n$.

\begin{lemma}\label{lm:sigma_bar_to_sigma_k}
Let $B$ be an $n\times n$ matrix. Then,
\begin{equation}\label{eq:sigma_bar_sigma_k}
\bar\sigma_k(B)=\sum_{m=0}^n\sum_{j=0}^m\binom{n-m}{k-j}(a-b)^{k-j}(a+b)^{n-m-k+j}\sigma_m(B).\end{equation}
\end{lemma}

\begin{proof}
Each product of the form $\lambda_{j_1}(B)\cdot\ldots\cdot\lambda_{j_l}(B)$ 
comes from expressions where $\lambda_{j_1}(B),\ldots,\lambda_{j_l}(B)$ 
are coupled with $j$ copies of $a-b$ and $j$ runs from 0 to $m$. 
For every such case, there are $\binom{n-m}{k-j}$ such expressions. 
Note that if $m>k$, then there are no terms which have more than $k$ copies of $a-b$ because then $k-j<0$. 
In each situation, it follows that $\lambda_{j_1}(B),\ldots,\lambda_{j_m}(B)$ 
must be coupled by $k-j$ copies of $(a-b)$ and $n-m-k+j$ copies of $(a+b)$. 
This yields the lemma when summing over $m$ from 0 to $n$.
\end{proof}

We now state our final combinatorial lemma.

\begin{lemma}\label{lm:linear_tau_pi/4_pi/2}
Let $A$ be given by \eqref{eq-expression-A} for $n\geq3$ and $a,b\in\mathbb R$. Then,
for each $i\in \{1, \cdots, n\}$, 
\[\begin{split}
&\Big(\sum_{0\leq 2k\leq n}(-1)^k\bar\sigma_{2k}(A)\Big)\\
&\quad\cdot
\Big(\sum_{1\leq 2k+1\leq n}(-1)^k\sum_{m=0}^n\sum_{j=0}^m
\binom{n-m}{2k+1-j}(a-b)^{2k+1-j}(a+b)^{n-m-2k-1+j}\hat \sigma_{m-1,i}(A)\Big)\\
&\quad-\Big(\sum_{1\leq 2k+1\leq n}(-1)^k\bar\sigma_{2k+1}(A)\Big)\\
&\quad\cdot
\Big(\sum_{0\leq 2k\leq n}(-1)^k\sum_{m=0}^n\sum_{j=0}^m
\binom{n-m}{2k-j}(a-b)^{2k-j}(a+b)^{n-m-2k+j}\hat \sigma_{m-1,i}(A)\Big)\\
&\quad=2^nb((\lambda_1+a)^2+b^2)\cdot\ldots\cdot((\lambda_{i-1}+a)^2+b^2)\\
&\qquad\qquad\cdot((\lambda_{i+1}+a)^2+b^2)\cdot\ldots\cdot((\lambda_n+a)^2+b^2).\end{split}\]
\end{lemma}

\begin{proof}
By Lemma \ref{lm:sigma_bar_to_sigma_k}, we have
\[\begin{split}&\sum_{m=0}^n\sum_{j=0}^l\binom{n-m}{2k+1-j}(a-b)^{2k+1-j}(a+b)^{n-m-2k-1+j}
\big(\sigma_m(A^{(i)})-\sigma_m(A_{(i)})\big)\\
&\qquad=\bar\sigma_{2k+1}(A^{(i)})-\bar\sigma_{2k+1}(A_{(i)}).\end{split}\]
We claim that
\begin{equation}\label{eq:sigma_bar_claim}\begin{split}
&\Big(\sum_{0\leq 2k\leq n}(-1)^k\bar\sigma_{2k}(A)\Big)
\Big(\sum_{1\leq 2k+1\leq n}(-1)^k\big(\bar\sigma_{2k+1}(A^{(i)})-\bar\sigma_{2k+1}(A_{(i)})\big)\Big)\\
&\qquad-\Big(\sum_{1\leq 2k+1\leq n}(-1)^k\bar\sigma_{2k+1}(A)\Big)
\Big(\sum_{0\leq 2k\leq n}(-1)^k\big(\bar\sigma_{2k}(A^{(i)})-\bar\sigma_{2k}(A_{(i)})\big)\Big)
\\&\quad=2^nb((\lambda_1+a)^2+b^2)\cdot\ldots\cdot((\lambda_{i-1}+a)^2+b^2)\\
&\qquad\quad\cdot((\lambda_{i+1}+a)^2+b^2)
\cdot\ldots\cdot((\lambda_n+a)^2+b^2).\end{split}\end{equation}

We prove \eqref{eq:sigma_bar_claim} by induction on $n$. 
We can verify \eqref{eq:sigma_bar_claim} for $n=3$ via a direct calculation. 
Now suppose that \eqref{eq:sigma_bar_claim} holds for $n\geq 3$ and we will prove that it must then hold for $n+1$. 
Set 
$$A=\diag(\lambda_1,\ldots,\lambda_{n+1}),$$
and 
$$\hat A=\diag(\lambda_1,\ldots,\lambda_{n}).$$ 
Then, $\hat A$ satisfies \eqref{eq:sigma_bar_claim} by the induction hypothesis. 

For simplicity, we set
\[\begin{split}
I(A)&=\Big(\sum_{0\leq 2k\leq n+1}(-1)^k\bar\sigma_{2k}(A)\Big)
\Big(\sum_{1\leq 2k+1\leq n+1}(-1)^k\big(\bar\sigma_{2k+1}(A^{(i)})-\bar\sigma_{2k+1}(A_{(i)})\big)\Big)\\
&\quad
-\Big(\sum_{1\leq 2k+1\leq n+1}(-1)^k\bar\sigma_{2k+1}(A)\Big)
\Big(\sum_{0\leq 2k\leq n+1}(-1)^k\big(\bar\sigma_{2k}(A^{(i)})-\bar\sigma_{2k}(A_{(i)})\big)\Big),\end{split}\]
and $I(\hat A)$ similarly. Since $\hat A$ is an $n\times n$ matrix, 
summations are up to $n$. In other words, we have 
\[\begin{split}
I(\hat A) &= \Big(\sum_{0\leq 2k\leq n}(-1)^k\bar\sigma_{2k}(\hat A)\Big)
\Big(\sum_{1\leq 2k+1\leq n}(-1)^k\big(\bar\sigma_{2k+1}(\hat A^{(i)})-\bar\sigma_{2k+1}(\hat A_{(i)})\big)\Big)\\
&\quad
-\Big(\sum_{1\leq 2k+1\leq n}(-1)^k\bar\sigma_{2k+1}(\hat A)\Big)
\Big(\sum_{0\leq 2k\leq n}(-1)^k\big(\bar\sigma_{2k}(\hat A^{(i)})-\bar\sigma_{2k}(\hat A_{(i)})\big)\Big).\end{split}\]
We will prove that
\begin{equation}\label{eq:I_equals_II-v2}I(A)=2((\lambda_{n+1}+a)^2+b^2)\,I(\hat A).\end{equation}
By the induction hypothesis, we have 
\[\begin{split}
I(\hat A)&=2^nb((\lambda_1+a)^2+b^2)\cdot\ldots\cdot((\lambda_{i-1}+a)^2+b^2)\\
&\qquad\qquad\cdot((\lambda_{i+1}+a)^2+b^2)\cdot\ldots\cdot((\lambda_n+a)^2+b^2).\end{split}\]
We then obtain the desired identity. 

Observe that for $A$ and $\hat A$ we have the following useful relation:
\begin{equation}\label{eq:sigma_bar_diag}
\bar\sigma_k(A)=(\lambda_{n+1}+a+b)\bar\sigma_k(\hat A)+(\lambda_{n+1}+a-b)\bar\sigma_{k-1}(\hat A).\end{equation}
It is not hard to see that
the relation \eqref{eq:sigma_bar_diag} also holds for $A^{(i)},\hat A^{(i)}$ and $A_{(i)},\hat A_{(i)}$. 
By \eqref{eq:sigma_bar_diag}, we write 
\begin{align*}\sum_{0\leq 2k\leq n+1}(-1)^k\bar\sigma_{2k}(A)
&=(\lambda_{n+1}+a+b) I_1(\hat A)+(\lambda_{n+1}+a-b) I_2(\hat A),\\
\sum_{1\leq 2k+1\leq n+1}(-1)^k\bar\sigma_{2k+1}(A)
&=(\lambda_{n+1}+a+b) I_3(\hat A)+(\lambda_{n+1}+a-b) I_4(\hat A),\end{align*}
where 
\begin{align*}I_1(\hat A)=\sum_{0\leq 2k\leq n}(-1)^k\bar\sigma_{2k}(\hat A),\quad 
&I_2(\hat A)=\sum_{0\leq 2k\leq n}(-1)^k\bar\sigma_{2k-1}(\hat A),\\
I_3(\hat A)=\sum_{1\leq 2k+1\leq n}(-1)^k\bar\sigma_{2k+1}(\hat A),\quad 
&I_4(\hat A)=\sum_{1\leq 2k+1\leq n}(-1)^k\bar\sigma_{2k}(\hat A). 
\end{align*}
Similar identities hold for $A^{(i)},\hat A^{(i)}$ and $A_{(i)},\hat A_{(i)}$. 
Then, 
\begin{equation}\label{eq-expression-I-hat-A}
I(\hat A)=I_1(\hat A)\big[I_3(\hat A^{(i)})-I_3(\hat A_{(i)})\big]
+I_2(\hat A)\big[I_4(\hat A^{(i)})-I_4(\hat A_{(i)})\big].
\end{equation}

Simple substitutions of \eqref{eq:sigma_bar_diag} and similar expressions for  $A^{(i)},\hat A^{(i)}$ and $A_{(i)},\hat A_{(i)}$ 
in $I(A)$ yield
\[\begin{split}I(A)&=\big[(\lambda_{n+1}+a+b)I_1(\hat A)+(\lambda_{n+1}+a-b)I_2(\hat A)\big]\\
&\qquad\cdot\big[(\lambda_{n+1}+a+b)I_3(\hat A^{(i)}) + (\lambda_{n+1}+a-b)I_4(\hat A^{(i)})\\
&\quad\qquad-(\lambda_{n+1}+a+b)I_3(\hat A_{(i)}) - (\lambda_{n+1}+a-b)I_4(\hat A_{(i)})\big]\\
&\qquad-\big[(\lambda_{n+1}+a+b)I_3(\hat A)+(\lambda_{n+1}+a-b)I_4(\hat A)\big]\\
&\qquad\quad\cdot\big[(\lambda_{n+1}+a+b)I_1(\hat A^{(i)}) + (\lambda_{n+1}+a-b)I_2(\hat A^{(i)})
\\&\qquad\quad\quad-(\lambda_{n+1}+a+b)I_1(\hat A_{(i)}) - (\lambda_{n+1}+a-b)I_2(\hat A_{(i)})\big].\end{split}\]
Note that here, the summation is now from $0$ or $1$ to $n$ rather than $n+1$. 
We now expand the expression in the right-hand side. 
First, after a rearrangement, terms involving $(\lambda_{n+1}+a)^2-b^2$ are given by 
\begin{align*}\big[(\lambda_{n+1}+a)^2-b^2\big]
\cdot&\big[\big(I_1(\hat A)I_4(\hat A^{(i)})-I_4(\hat A)I_1(\hat A^{(i)}\big)\\
&\quad-\big(I_1(\hat A)I_4(\hat A_{(i)})-I_4(\hat A)I_1(\hat A_{(i)})\big)\\
&\quad+\big(I_2(\hat A)I_3(\hat A^{(i)})-I_3(\hat A)I_2(\hat A^{(i)})\big)\\
&\quad-\big(I_2(\hat A)I_3(\hat A_{(i)})-I_3(\hat A)I_2(\hat A_{(i)}\big)\big].
\end{align*}
We note that each difference above is zero. 
The terms in our sum which have only even or odd-degree $\bar\sigma_k$-terms cancel. 
We arrange the remaining parts of $I(A)$ and write 
\[\begin{split}
I(A)&=\big[(\lambda_{n+1}+a+b)^2I_1(\hat A)I_3(\hat A^{(i)})-(\lambda_{n+1}+a-b)^2I_4(\hat A)I_2(\hat A^{(i)})\big]\\
&\qquad-\big[(\lambda_{n+1}+a+b)^2I_1(\hat A)I_3(\hat A_{(i)})-(\lambda_{n+1}+a-b)^2I_4(\hat A)I_2(\hat A_{(i)})\big]\\
&\qquad+\big[(\lambda_{n+1}+a-b)^2I_2(\hat A)I_4(\hat A^{(i)})+(\lambda_{n+1}+a+b)^2I_3(\hat A)I_1(\hat A_{(i)})\big]\\
&\qquad-\big[(\lambda_{n+1}+a-b)^2I_2(\hat A)I_4(\hat A_{(i)})-(\lambda_{n+1}+a+b)^2I_3(\hat A)I_1(\hat A_{(i)})\big].\end{split}\]
Each difference consists of like terms. In fact, we have 
\begin{align*}
&I_1(\hat A)I_3(\hat A^{(i)})=-I_4(\hat A)I_2(\hat A^{(i)}), \quad
I_1(\hat A)I_3(\hat A_{(i)})=-I_4(\hat A)I_2(\hat A_{(i)}), \\
&I_2(\hat A)I_4(\hat A^{(i)})=-I_3(\hat A)I_1(\hat A_{(i)}), \quad
I_2(\hat A)I_4(\hat A_{(i)})=-I_3(\hat A)I_1(\hat A_{(i)}).
\end{align*}
Hence, 
\begin{align*}
I(A)=2((\lambda_{n+1}+a)^2+b^2)&\big[
I_1(\hat A)I_3(\hat A^{(i)})
-I_1(\hat A)I_3(\hat A_{(i)})\\
&\qquad+I_2(\hat A)I_4(\hat A^{(i)})
-I_2(\hat A)I_4(\hat A_{(i)})\big].
\end{align*}
By \eqref{eq-expression-I-hat-A}, we obtain \eqref{eq:sigma_bar_claim}.
\end{proof}

\section{Special Lagrangian Equations}\label{sec-slag-dim-n}

In this section, we study the special Lagrangian equation in arbitrary dimension 
and prove Theorem \ref{thrm-slag-expansion} and Theorem \ref{new-slag-odd-dim}. 
We will consider only the case $n\ge 3$. 

It is well-known that the special Lagrangian equation takes the algebraic form
\begin{equation}\label{eq:4-dim-n}\cos(\theta)\sum_{1\leq 2k+1\leq n}(-1)^k\sigma_{2k+1}(\nabla^2u) 
- \sin(\theta)\sum_{0\leq 2k\leq n}(-1)^k\sigma_{2k}(\nabla^2u) = 0.\end{equation}
Let $u$ be a smooth solution of \eqref{eq:4-dim-n} in $\mathbb R^n\setminus\bar B_1$. 
Under the assumption of the supercritical phase or the semiconvexity as in Theorem \ref{thrm-slag-expansion}, 
Li, Li, and Yuan \cite{LLY2020} proved 
that there exist an $n\times n$ symmetric matrix $A$, a vector $b\in\mathbb R^n$, and a constant $c\in\mathbb R$ such that, 
for any $x\in\mathbb R^n\setminus\bar B_1$ and any integer $\ell\ge 0$, 
\begin{equation}\label{eq:special_lagrangian_sol-dim-n}
u(x) = \frac{1}{2}x^TAx+ b\cdot x + c + O_\ell(|x|^{2-n})\quad\text{as } x\to\infty.\end{equation}
Here, the notation $O_\ell(|x|^{n-2})$ means that
\begin{equation}\label{eq:O_ell-dim-n}
\limsup_{x\to\infty}|x|^{n-2+\ell}\Big|\nabla^\ell\Big[u(x)-\frac{1}{2}x^TAx+b\cdot x + c\Big]\Big|<\infty.\end{equation}
By \eqref{eq:special_lagrangian_sol-dim-n}, it follows that $A$ also satisfies
\[\cos(\theta)\sum_{1\leq 2k+1\leq n}(-1)^k\sigma_{2k+1}(A) - \sin(\theta)\sum_{0\leq 2k\leq n}(-1)^k\sigma_{2k}(A)  = 0.\]
A simple algebraic manipulation yields 
\begin{equation}\label{eq:special_lagrangian_no_theta-dim-n}\begin{split}
&\bigg(\sum_{0\leq 2k\leq n}(-1)^k\sigma_{2k}(A)\bigg)
\bigg(\sum_{1\leq 2k+1\leq n}(-1)^k\sigma_{2k+1}(\nabla^2u)\bigg) \\ 
&\qquad-\bigg(\sum_{1\leq 2k+1\leq n}(-1)^k\sigma_{2k+1}(A)\bigg)
\bigg(\sum_{0\leq 2k\leq n}(-1)^k\sigma_{2k}(\nabla^2u)\bigg) = 0.\end{split}\end{equation}
We assume without a loss of generality that the matrix $A$ is diagonal. 
In the most general case, $A$ can be a symmetric matrix, 
but we can always diagonalize it by an appropriate rotation. 
Throughout the following discussion, we assume 
\begin{align}\label{eq-expression-A-slag}A=\diag(\lambda_1,\ldots,\lambda_n).\end{align}
We now start to prove 
Theorem \ref{thrm-slag-expansion} and Theorem \ref{new-slag-odd-dim}. 

\begin{proof}[Proof of Theorem \ref{thrm-slag-expansion} for $n\ge 3$.]
The proof consists of three steps. In the first step, we introduce a function $v$ in the punctured unit disc by 
the generalized Kelvin transform to rewrite the solution $u$ and demonstrate that 
the equation satisfied by $v$ is a perturbation of the Laplace equation. 
In the second step, we study the regularity of the perturbative terms. 
In the third step, we prove that $v$ can be extended to a function with a higher regularity across the origin. 

{\it Step 1.} For some diagonal matrix $R$ to be determined, we consider 
\begin{equation}\label{eq-y-slag}y=\frac{R x}{|R x|^2}.\end{equation}
In view of \eqref{eq:special_lagrangian_sol-dim-n}, we introduce a function $v=v(y)$ in $B_1\setminus\{0\}$ such that 
\begin{equation}\label{eq:special_lagrangian_kelvin}u(x) = \frac{1}{2}x^TAx+ b\cdot x + c + |y|^{n-2}v(y).\end{equation}
We can verify by a direct computation that
\begin{equation}\label{eq:uij}u_{ij}(x) = A_{ij} + |y|^n N_{ij}(v),\end{equation}
where
\begin{equation}\label{eq:mij_tilde}\begin{split}
N_{ij}(v) &= \big[n(n-2)v+4n\langle y,\nabla v\rangle+4y_ky_l v_{kl}\big]R_{ii}R_{jj}y_iy_j|y|^{-2} 
\\&\qquad-\big[(n-2)v+2\langle y,\nabla v\rangle\big]R_{ii}R_{jj}\delta_{ij}- nR_{ii}R_{jj}(y_iv_j+y_jv_i)
\\&\qquad-2R_{ii}R_{jj}(y_iy_kv_{kj} + y_jy_m v_{im})+R_{ii}R_{jj}|y|^2v_{ij}.\end{split}\end{equation}
Here, the derivatives of $v$ are taken with respect to $y$
and $R_{ii},R_{jj}$ are entries of the matrix $R$. 
We should point out that $i$ and $j$ are fixed in \eqref{eq:mij_tilde}.
Observe that we can write the matrix $N(v)$ as a sum of rank-one matrices and the Hessian of $v$
as follows: 
\begin{equation}\label{eq:altM}\begin{split}
N(v) &= \big[n(n-2)v+4n\langle y,\nabla v\rangle+4y_ky_l v_{kl}\big]|y|^{-2}(Ry)\cdot (Ry)^T \\
&\qquad -\big[(n-2)v+2\langle y,\nabla v\rangle\big]R^2-\big[Ry\cdot(R\nabla v)^T
+R\nabla v\cdot (Ry)^T\big]
\\ &\qquad-2\big[Ry\cdot(\nabla^2v\cdot Ry)^T+(\nabla^2v\cdot Ry)\cdot (Ry)^T\big] + |y|^2R \nabla^2vR.
\end{split}\end{equation}
It will be useful to write $N(v)$ as 
\begin{equation}\label{eq:mij_tilde-v2}
N_{ij}(v) = R_{ii}R_{jj}M_{ij}(v),\end{equation}
where 
\begin{equation}\label{eq:mij}\begin{split}
M_{ij}(v) &= \big[n(n-2)v+4n\langle y,\nabla v\rangle+4y_ky_l v_{kl}\big]y_iy_j|y|^{-2} \\
&\qquad-\big[(n-2)v+2\langle y,\nabla v\rangle\big]\delta_{ij}- n(y_iv_j+y_jv_i)\\
&\qquad-2(y_iy_kv_{kj} + y_jy_m v_{im})+|y|^2v_{ij}.\end{split}\end{equation}
It is easy to check that 
\begin{equation}\label{eq-trace-M}\Tr(M(v))=|y|^2\Delta v.\end{equation}
We omit the dependence on $v$ from our notations for the matrices $M, N$ from now on.

By \eqref{eq:O_ell-dim-n}, we have, for any integer $\ell\geq0$, 
\[\sup_{B_1\setminus\{0\}}|y|^\ell |\nabla^\ell v(y)|<\infty,\]
and in particular, $v$ is bounded in $B_1\setminus\{0\}$. 
Furthermore, in each term of $M_{ij}$, every $v_i$ is coupled with a factor $y_j$ 
and every $v_{ij}$ is coupled with at least a factor $y_ky_m$ 
and so $M_{ij}$ is a bounded function even though the factor $|y|^{-2}y_iy_j$ is not continuous at $y=0$.


We now rewrite \eqref{eq:special_lagrangian_no_theta-dim-n} in terms of $v$ using \eqref{eq:uij} to get
\begin{equation}\label{eq:special_lagrangian_v}\begin{split}
&\bigg(\sum_{0\leq 2k\leq n}(-1)^k\sigma_{2k}(A)\bigg)
\bigg(\sum_{1\leq 2k+1\leq n}(-1)^k\sigma_{2k+1}(A+|y|^nN)\bigg) \\
&\qquad-\bigg(\sum_{1\leq 2k+1\leq n}(-1)^k\sigma_{2k+1}(A)\bigg)
\bigg(\sum_{0\leq 2k\leq n}(-1)^k\sigma_{2k}(A+|y|^nN)\bigg) = 0.\end{split}\end{equation}
We will rewrite \eqref{eq:special_lagrangian_v} according to the orders of $v$, $\nabla v$, and $\nabla^2v$. 
First, zero-order terms all cancel out. 
Second, by the definition of $N$ in \eqref{eq:mij_tilde}, 
we observe that the only parts of \eqref{eq:special_lagrangian_v} 
linear in $v$ and their derivatives are the terms of \eqref{eq:special_lagrangian_v} 
which are linear in $N_{ij}$. 
By Lemma \ref{lm:sigma_k_to_sigma_hat}, 
we see that the terms of $\sigma_k(A+|y|^nN)$ which are linear in $N_{ij}$ have the form
\[|y|^n\sum_{i = 1}^n\hat \sigma_{k-1,i}(A)N_{ii}.\]
Hence, the $N_{ij}$-linear part of \eqref{eq:special_lagrangian_v} has the form
\begin{equation}\label{eq:linear_in_mij}\begin{split}
|y|^n\sum_{i=1}^n&\,\bigg[\bigg(\sum_{0\leq 2k\leq n}(-1)^k\sigma_{2k}(A)\bigg)
\bigg(\sum_{1\leq 2k+1\leq n}(-1)^k\hat \sigma_{2k,i}(A)\bigg) \\
&\qquad-\bigg(\sum_{1\leq 2k+1\leq n}(-1)^k\sigma_{2k+1}(A)\bigg)
\bigg(\sum_{0\leq 2k\leq n}(-1)^k\hat \sigma_{2k-1,i}(A)\bigg)\bigg]N_{ii}\\
&=|y|^n\sum_{i=1}^n(1+\lambda_1^2)\cdot\ldots\cdot(1+\lambda_{i-1}^2)(1+\lambda_{i+1}^2)
\cdot\ldots\cdot(1+\lambda_n^2)N_{ii},\end{split}\end{equation}
where we used Lemma \ref{lm:special_lagrangian_linear}. 
Now, we take 
\begin{equation}\label{eq-expression-R-dim-n}
R=\diag\big((1+\lambda_1^2)^{1/2}, \cdots, (1+\lambda_n^2)^{1/2}\big).\end{equation}
Since we assume $A$ is diagonal, we once again have $R=\sqrt{I+A^2}$, as in Section 2.

By \eqref{eq:mij_tilde-v2}, we have 
$$N_{ii}=(1+\lambda_{i}^2)M_{ii}.$$
By \eqref{eq-trace-M}, we can rewrite \eqref{eq:linear_in_mij} as
\begin{align*}
|y|^n(1+\lambda_1^2)\cdot\ldots\cdot(1+\lambda_n^2)\sum_{i=1}^nM_{ii}
=|y|^{n+2}\det(R^2)\Delta v.
\end{align*}
This is the only term linear in $v$, $\nabla v$, and $\nabla^2v$ in the equation \eqref{eq:special_lagrangian_v}. 
We point out that $\det(R^2)>0$. 
We next examine the higher order terms. 

For each $k=2, \cdots, n$, $\sigma_{k}(A+|y|^nN)$ is the sum of all $k\times k$ principal minors of $A+|y|^nN$, 
which can be written as a linear combination of $m\times m$ principal minors of $|y|^nN$, for $m=0, 1, \cdots, k$, 
with coefficients given by polynomials of $\lambda_1, \cdots, \lambda_n$. 
We already discussed the cases $m=0$ and $m=1$. We now consider $2\le m\le k$ for any $2\le k\le n$. 
We can write \eqref{eq:special_lagrangian_v} as 
\begin{align*}|y|^{n+2}\det(R^2)\Delta v+\sum_{m=2}^n|y|^{m n}J_m=0,\end{align*}
where $J_m$ is a linear combination of $m\times m$ principal minors of $N$. Hence, 
\begin{align}\label{eq:special_lagrangian_v-dim-n}\Delta v+\det(R^{-2})\sum_{m=2}^n|y|^{m n-n-2}J_m=0.\end{align}

{\it Step 2.} We point out that each $J_m$ is a homogeneous polynomial of $v$, $\nabla v$, and $\nabla^2v$, of degree $m$, 
with coefficients given by functions of $y$. 
We now study the regularity of these functions. 
We will 
show that $|y|^{m n-n-2}J_m$ is either a smooth or $C^{n-3,\alpha}$ function of $y$ in $B_1$, for any $\alpha\in (0,1)$. 
Note that in the expression of $N_{ij}$ in \eqref{eq:mij_tilde}, 
there is a singular factor $|y|^{-2}$. 

We already studied the equation \eqref{eq:special_lagrangian_v-dim-n} for $n=3$ in Section \ref{sec-slag-dim3}. 
By \eqref{eq:14-dim3}, we have 
$$\Delta v+ |y|^{-1}\sum_{i,j=1}^3y_iy_jF_{ij}+|y|^2\sum_{i,j=1}^3y_iy_jG_{ij}=0,$$
where $F_{ij}$ and $G_{ij}$ are polynomials in $y$, $v$, $\nabla v$, and $\nabla^2v$. 
In fact, $F_{ij}$ are homogeneous quadratic in $v$, $\nabla v$, and $\nabla^2v$, 
and $G_{ij}$ are homogeneous cubic in $v$, $\nabla v$, and $\nabla^2v$.  
The coefficients $y_iy_j|y|^{-1}$ have homogeneous degree 1 and are Lipschitz in $B_1$. 
In fact, they are the only nonsmooth factors. 

Now we move on to the general case for $n\geq4$. 
Recall that $N_{ij}$ can be expressed in terms of $M_{ij}$ by \eqref{eq:mij_tilde-v2}. 
In view of \eqref{eq:mij}, we write 
\begin{equation*}
M_{ij} =K_{ij}+ \frac{y_iy_j}{|y|^{2}}L,\end{equation*}
where 
\begin{align}\label{eq-expressions-K-ij-dim-n}\begin{split}
K_{ij}&=-\big[(n-2)v+2\langle y,\nabla v\rangle\big]\delta_{ij}- n(y_iv_j+y_jv_i)\\
&\qquad-2(y_iy_kv_{kj} + y_jy_m v_{im})+|y|^2v_{ij},\end{split}\end{align}
and 
\begin{align}\label{eq-expressions-L-dim-n}L=n(n-2)v+4n\langle y,\nabla v\rangle+4y_ky_m v_{km}.\end{align}
We point out that $K_{ij}$ and $L$ are linear in $v$, $\nabla v$, and $\nabla^2v$.

We first consider $|y|^{n-2}J_2$ in \eqref{eq:special_lagrangian_v-dim-n}, corresponding to $m=2$. 
Recall that $J_2$ is a linear combination of $2\times 2$ principal minors of $N$. 
We now consider, for $1\le i<j\le n$, 
$$N_{ii}N_{jj}-N_{ij}^2=R^2_{ii}R^2_{jj}(M_{ii}M_{jj}-M_{ij}^2).$$
A straightforward computation yields 
\begin{align*}
M_{ii}M_{jj}-M^2_{ij}&=(K_{ii}K_{jj}-K^2_{ij})+\frac{L}{|y|^2}(y_i^2K_{jj}+y_j^2K_{ii}-2y_iy_jK_{ij})\\
&=\frac{1}{|y|^2}\big[|y|^2(K_{ii}K_{jj}-K^2_{ij})+L(y_i^2K_{jj}+y_j^2K_{ii}-2y_iy_jK_{ij}).\end{align*}
Hence, 
$$|y|^{n-2}J_2=|y|^{n-4}\sum_{i,j=1}^ny_iy_jF^{(2)}_{ij},$$
where $F^{(2)}_{ij}$ are homogeneous quadratic polynomials in $v$, $\nabla v$, and $\nabla^2v$, 
with coefficients given by polynomials of $y$. We point out that there are no negative powers of $|y|$ in $F^{(2)}_{ij}$. 
Similarly, for each $m=3, \cdots, n$, we have 
$$|y|^{m n-n-2}J_m=|y|^{m n-n-4}\sum_{i,j=1}^ny_iy_jF^{(m)}_{ij},$$ 
where $F^{(m)}_{ij}$ are homogeneous polynomials in $v$, $\nabla v$, and $\nabla^2v$ of degree $m$, 
with coefficients given by polynomials of $y$. 
We point out that there are no negative powers of $|y|$ in $F^{(m)}_{ij}$. 
As a consequence, \eqref{eq:special_lagrangian_v-dim-n} has the form 
\begin{align}\label{eq:special_lagrangian_v-dim-n-v2}\Delta v+\sum_{m=2}^n|y|^{m n-n-4}\sum_{i,j=1}^ny_iy_jF^{(m)}_{ij}=0,\end{align}
if we rename $F^{(m)}_{ij}$.

We now analyze the powers in $|y|^{m n-n-4}$. For $m=3, \cdots, n$, we have 
$$m n-n-4\ge 2n-4\ge 2,$$
by $n\ge 3$. For $m=2$, we have 
$$2n-n-4= n-4.$$
This is the smallest power of $|y|$ in \eqref{eq:special_lagrangian_v-dim-n-v2}. 
For $n=3$ and $m=2$, $m n-n-4=-1$, the only negative power of $|y|$. 
Moreover, for $n\ge 4$ even, $m n-n-4$ is even for each $m=2, \cdots, n$. 
If $n\ge 3$ is odd, $m n-n-4$ is even for odd $m\ge 3$ and is odd for even $m\ge 4$. 
The lowest-regularity terms are the ones with coefficients $y_iy_j|y|^{n-4}$ given by $m=2$, which are $C^{n-3,1}$ in $B_1$ 
when $n$ is odd, and smooth when $n$ is even. 

{\it Step 3.} We now rewrite \eqref{eq:special_lagrangian_v-dim-n-v2} as 
\begin{equation}\label{eq:desired_form-dim-n}\Delta v + f(v;y) = 0\quad\text{in }B_1\setminus\{0\},\end{equation}
where $f(v;y)$ is a polynomial of $v$, $\nabla v$, and $\nabla^2v$, 
and is $C^{n-3,1}$ in $y$ when $n$ is odd and  is smooth in $y$ when $n$ is even. 
We now prove that we can extend $v$ to either a $C^{n-1,\alpha}(B_1)$ function for any $\alpha\in(0,1)$ when $n$ is odd, 
or to a $C^\infty(B_1)$ function when $n$ is even. 
We prove this using a perturbation argument from \cite{HanWang}, which we present below.

Let $\varepsilon$ be a small positive constant to be determined. 
Set $z=\varepsilon^{-1}y$ and 
$$w(z)=v(y)\quad\text{for }z\in B_1.$$ 
Then, $y_jv_i=z_jw_i$ and $y_ky_m v_{ij}=z_kz_m w_{ij}$. 
By \eqref{eq:special_lagrangian_v-dim-n-v2}, we know that $w$ satisfies
\begin{equation}\label{eq:w_equation}\Delta w + h(w,z,\varepsilon)=0\quad\text{in }B_1\setminus\{0\},\end{equation}
where
\begin{equation}\label{eq:h_def}h(w,z,\varepsilon)=\sum_{k=2}^n\varepsilon^{kn-n}|z|^{kn-n-2}P_k(w,z,\varepsilon).\end{equation}
Here, $P_k$ are polynomials of $\varepsilon,z,w,\nabla w$, and $\nabla^2w$, 
but their expressions are of no particular interest to us beyond that. 
The equation \eqref{eq:w_equation} is a fully nonlinear uniformly elliptic equation for sufficiently small $\varepsilon$. 

Now fix $\alpha\in(0,1)$. We claim that for sufficiently small $\varepsilon$, there exists $\zeta\in C^{2,\alpha}(\bar B_1)$ such that
\begin{align}\label{eq:zeta_equation}\begin{split}
\Delta\zeta + h(\zeta,z,\varepsilon)&=0\quad\text{in }B_1,\\ 
\zeta&=w\quad\text{on }\partial B_1,\end{split}\end{align}
and moreover, the $C^{2,\alpha}(\bar B_1)$-norm of $\zeta$ is uniformly bounded in $\varepsilon$. 
Assuming our claim is true, it follows that $w=\zeta$ in $B_1\setminus\{0\}$ by the uniqueness of $W^{2,p}$-solutions 
and hence $w$ can be extended to a $C^{2,\alpha}(B_1)$ function. 
Because $h$ is a smooth function of $w,\nabla w, \nabla^2w$ 
and is a smooth function in $z$ for even $n$ and $C^{n-3,\alpha}$ for odd $n$, 
it follows by a bootstrap argument that $w$ itself is smooth for even $n$ and $C^{n-1,\alpha}$ for odd $n$ and hence so is $v$.

It remains to prove our claim and solve \eqref{eq:zeta_equation}. 
Let $\tilde w\in C^\infty(\bar B_1)$ be such that $\tilde w = w$ on $\partial B_1$ and set
\[\mathcal X=\{\psi\in C^{2,\alpha}(\bar B_1):\psi = 0\text{ on }\partial B_1\}.\]
For $\varepsilon_0>0$ sufficiently small (to be determined later), 
define $F:\mathcal X\times[0,\varepsilon_0]\to C^\alpha(\bar B_1)$ by
\[F(\psi,\varepsilon)=\Delta\psi + h(\psi+\tilde w,z,\varepsilon)+\Delta\tilde w.\]
The Fr\'echet derivative of $F$ with respect to $\psi$ is then
\[F_\psi(\psi,\varepsilon)\varphi=\Delta\varphi+\frac{d}{dt}\Big|_{t=0}h(\psi+\tilde w+t\varphi,z,\varepsilon).\]
Since we have positive powers of $\varepsilon$ in each term from the second term of \eqref{eq:h_def}, we get
\[F_\psi(\psi,0)\varphi=\Delta\varphi,\]
and hence for any $\psi\in\mathcal X$, the operator $F_\psi(\psi,0)$ is an invertible operator 
from $\mathcal X$ to $C^{\alpha}(\bar B_1)$. Take $\psi_0\in\mathcal X$ such that $\Delta\psi_0=-\Delta\tilde w$, 
which we know must exist by the Schauder theory. 
By the implicit function theorem, $F(\psi,\varepsilon)=0$ admits a solution $\psi=\psi_\varepsilon$ 
for sufficiently small $\varepsilon>0$ and this solution $\psi_\varepsilon$ is close to $\psi_0$ 
in the $C^{2,\alpha}(\bar B_1)$-norm. Hence, $\zeta = \psi_\varepsilon+\tilde w$ solves \eqref{eq:zeta_equation}.
\end{proof} 

In the rest of this section, we prove Theorem \ref{new-slag-odd-dim}. 
Before proceeding, we state and prove a simple lemma concerning homogeneous polynomials.

\begin{lemma}\label{lm:homogeneous_polys}
For any $m\in\mathbb Z_+$ and any homogeneous polynomial $h$ of degree $m$, 
there is a homogeneous polynomial $u$ of degree $m$ such that
\[\Delta(|y|^{n-2}u) = |y|^{n-4}h.\]
\end{lemma}

\begin{proof}
We present the proof for $n=3$ in details. 
We do this since this lemma will ultimately be used to prove Theorem \ref{new-slag-odd-dim}, 
where $n=3$ serves as our model case for the other odd dimensions. 
The proof for higher dimensions follows in the same way. It is easy to compute that
\begin{equation}\label{eq:laplacepoly}\Delta(|y|u) = |y|^{-1}(4u+|y|^2\Delta u),\end{equation}
where we used the assumption that $u$ is a homogeneous polynomial. 
If $m=0,1$, then $\Delta u =0$, so our claim reduces to finding $u$ such that $4u=h$. 
In general, for $m\in\mathbb Z_+$, we assume that the result holds for $m-2$. As in \eqref{eq:laplacepoly}, we have
\[\Delta(|y|h) = 4h|y|^{-1}+|y|\Delta h,\]
where $\Delta h$ is now a homogeneous polynomial of degree $m-2$. 
It is known (by Lemma 3.3 in \cite{HanWang}) that there exists a homogeneous polynomial $|y|^2\widetilde u$ 
of degree $m-2$ such that $\Delta(|y|^3\widetilde u) = |y|\Delta h$. Therefore,
\[\Delta(|y|h-|y|^3\widetilde u) = 4h|y|^{-1}.\]
Hence, we take $u = (h-|y|^4\widetilde u)/4$.

For general dimensions, we proceed in the same way, but now we start with
\[\Delta(|y|^{n-2}u) = |y|^{n-4}((3n-5)u+|y|^2\Delta u),\]
instead of \eqref{eq:laplacepoly}.
\end{proof}

We now prove a preliminary decomposition of $v$ based on \eqref{eq:special_lagrangian_v-dim-n-v2}. 

\begin{lemma}\label{lm:taylor_expns}
Let $n\ge 3$ be an odd integer and $v$ be as in Theorem \ref{thrm-slag-expansion}. 
Then, for any $\ell\geq n-2$ there exist a polynomial $q_{\ell-n+2}$ of degree $\ell-n+2$, 
with $q_{\ell-n+2}(0)=0$ and $\nabla q_{\ell-n+2}(0)=0$, 
and a function $v_\ell\in C^{\ell,\alpha}(B_1)$ for any $\alpha\in(0,1)$ such that, for any $y\in B_1$,
\[v(y) = |y|^{n-2}q_{\ell-n+2}(y)+v_\ell(y).\]
\end{lemma}

\begin{proof}
We will prove this lemma inductively for $\ell$. 
The proof proceeds in the same way as the proof of Lemma 3.3 in \cite{HanWang}. 
We will mainly focus on differences caused by the lower degree in $|y|^{n-4}$ in \eqref{eq:special_lagrangian_v-dim-n-v2}.
We first write $v$ as the sum of a polynomial $p$,  the product of $|y|^{n-2}$ and a polynomial $q$, 
and a remainder term $w$. We then rewrite equation (4.20) into an equation for $w$ and, 
by choosing $p$ and $q$ appropriately, use the Schauder theory to conclude that $w\in C^{\ell,\alpha}(B_1)$. 

\textit{Step 1: Rewriting \eqref{eq:special_lagrangian_v-dim-n-v2}.}
Let $p$ and $q$ be two polynomials, 
with $q(0)=0$ and $\nabla q(0)=0$. 
The degrees of $p$ and $q$ do not matter to us at this point. We define $w$ by
\begin{align}\label{eq-def-w}v = w + p + |y|^{n-2}q.\end{align}
We now substitute this in \eqref{eq:special_lagrangian_v-dim-n-v2}. 
Note that the summation for $m$ in \eqref{eq:special_lagrangian_v-dim-n-v2} is from 2 to $n$. 
It is obvious that $m n-n-4$ is odd if $m$ is even, with a minimum $n-4$ for $m=2$, 
and that $m n-n-4$ is even if $m$ is odd, with a minimum $2n-4$ for $m=3$. 
We can write \eqref{eq:special_lagrangian_v-dim-n-v2} as an equation of $w$ in the form 
\begin{equation}\label{eq:w_eq}
\Delta w + \Delta p + \Delta(|y|^{n-2}q)+ |y|^{2n-4}\mathcal F(w;y,p,q)+|y|^{n-4}\mathcal G(w;y,p,q) =0,\end{equation}
where $\mathcal F,\mathcal G$ are nonlinear second-order differential operators which depend smoothly on $y,w$, 
and the derivatives of $w$ up to order 2. 
We point out that each term in $\mathcal F,\mathcal G$ has a factor $y_iy_j$.
Moreover, $\mathcal F,\mathcal G$ are polynomials of $w$, $\nabla w$, and $\nabla^2 w$, 
with degrees varying from 0 to $n$. The degree 0 parts are given by 
$\mathcal F(0;y,p,q)$ and $\mathcal G(0;y,p,q)$, respectively. 
We need to point out that neither $\mathcal F(0;y,p,q)$ nor $\mathcal G(0;y,p,q)$ is a polynomial in $y$, 
due to the presence of $|y|^{n-2}$ in $p + |y|^{n-2}q$. 
We write
\[p = \sum_{i=0}^{\deg p}P_i,\quad q=\sum_{i=0}^{\deg q}Q_i,\]
for homogeneous polynomials $P_i,Q_i$ of degree $i$, with 
$$Q_0=Q_1=0.$$ 
Then, we can write
\begin{align}\label{eq-decom-F-G-0} |y|^{2n-4}\mathcal F(0;y,p,q)+|y|^{n-4}\mathcal G(0;y,p,q) = 
\sum_{k\geq0}\widetilde P_k+|y|^{n-4}\sum_{k\geq0}\widetilde Q_k,
\end{align} 
for homogeneous polynomials $\widetilde P_k,\widetilde Q_k$ of degree $k$, with 
$$\widetilde P_i=0\text{ for }0\le i\le 2n-3\quad\text{and}\quad\widetilde Q_0=\widetilde Q_1=0.$$
The summations in \eqref{eq-decom-F-G-0} are finite and their order depends on $\deg p$ and $\deg q$. 
Then, we can see that for each $k\geq2$, 
$\widetilde Q_k$ is determined by $P_i, Q_j$ for $i,j\leq  k-2$. 
For the term $v^2$ in $\mathcal G(0;y,p,q)$, for example, we have
\begin{align*}
&|y|^{n-4}\bigg(\sum_{i\geq0}P_i+ |y|^{n-2}\sum_{i\geq0}Q_i\bigg)^2 \\
&\qquad= 2 |y|^{2n-6}\sum_{i,j\geq0}P_iQ_j
+|y|^{n-4}\bigg[\sum_{i,j\geq0}P_iP_j+|y|^{2n-4}\sum_{i,j\geq0}Q_iQ_j\bigg].\end{align*}
The first term in the right-hand side is a polynomial of $y$, 
and hence a part of the first summation in the right-hand side of \eqref{eq-decom-F-G-0}. 
The expression in the parentheses in the second term  
is a part of the second summation in the right-hand side of \eqref{eq-decom-F-G-0}. 
Since each term in $\mathcal G$ contains a factor $y_{k}y_{m}$, 
the corresponding part of 
the second summation in the right-hand side of \eqref{eq-decom-F-G-0} is given by
\[\bigg(\sum_{km=1}^nc_{km}y_{k}y_{m}\bigg)\cdot\bigg(\sum_{i,j\geq0}P_iP_j+|y|^{2n-4}\sum_{i,j\geq0}Q_iQ_j\bigg),\]
and hence the homogeneous part of degree $k$ is given by
\[\bigg(\sum_{km=1}^nc_{km}y_{k}y_{m}\bigg)\cdot\bigg(\sum_{i+j=k-2}P_iP_j+|y|^{2n-4}\sum_{i+j=k-2n+2}Q_iQ_j\bigg),\]
for some constants $c_{km}$.
We need to point out that both terms  in the left-hand side of \eqref{eq-decom-F-G-0} 
contribute to the decomposition in the right-hand side.

\textit{Step 2: The initial step.}
By Theorem \ref{thrm-slag-expansion}, we have $v\in C^{n-1,\alpha}(B_1)$. 
Let $p_{n-1}$ be the degree $n-1$ Taylor polynomial of $v$ and $r_{n-1}$ be the remainder. 
Then, 
$$v=p_{n-1}+r_{n-1}\quad\text{in }B_1,$$ 
and 
$$r_{n-1}\in C^{n-1,\alpha}(B_1)\quad\text{and}\quad \nabla^kr_{n-1}(0)=0 \text{ for }k=0,\ldots,n-1.$$ 
In this case, we can take $q_{1}=0$  and have 
$$v=p_{n-1}+|y|^{n-2}q_1+r_{n-1}\quad\text{in }B_1.$$ 
This is our base of induction.

\textit{Step 3: The induction.}
Now take an integer $\ell\geq n$ and suppose that there exist polynomials 
$p_{\ell-1},q_{\ell-n+1}$ of degrees $\ell-1$ and $\ell-n+1$, respectively, 
such that 
$$v=p_{\ell-1}+ |y|^{n-2}q_{\ell-n+1}+r_{\ell-1}\quad\text{in }B_1,$$ 
with 
$$q_{\ell-n+1}(0)=0,\quad \nabla q_{\ell-n+1}(0)=0,$$
and 
$$r_{\ell-1}\in C^{\ell-1,\alpha}(B_1)\quad\text{and}\quad  \nabla^kr_{\ell-1}(0)=0\text{ for }k=0,\ldots,\ell-1.$$
We will find homogeneous polynomials $P_\ell$ and $Q_{\ell-n+2}$ and an appropriate function $r_\ell\in C^{\ell,\alpha}(B_1)$, 
with $\nabla^kr_{\ell}(0)=0\text{ for }k=0,\ldots,\ell$, such that 
\[v=(p_{\ell-1}+P_\ell)+|y|^{n-2}(q_{\ell-n+1}+Q_{\ell-n+2})+r_\ell.\] 
By setting 
$$p_\ell= p_{\ell-1}+P_\ell, \quad q_{\ell-n+2}= q_{\ell-n+1}+Q_{\ell-n+2},$$ 
we conclude
\[v=p_\ell+|y|^{n-2}q_{\ell-n+2}+r_\ell.\] 
Step 2 corresponds to $\ell=n$ in the induction hypothesis.

With $p_{\ell-1},q_{\ell-n+1}$ as in the induction hypothesis, we write 
\[p_{\ell-1} = \sum_{i=0}^{\ell-1}P_i,\quad q_{\ell-n+1}=\sum_{i=0}^{\ell-n+1}Q_i,\]
for homogeneous polynomials $P_i,Q_i$ of degree $i$, with $Q_0=Q_1=0.$ 
For a homogeneous polynomial $Q_{\ell-n+2}$ to be determined, we set
\[q_{\ell-n+2} = q_{\ell-n+1}+Q_{\ell-n+2}.\]
For brevity of notations, we write $p$ in place of $p_{\ell-1}$ and $q$ in place of $q_{\ell-n+2}$, 
i.e., $p=p_{\ell-1}$ and $q=q_{\ell-n+2}$. Then, 
\begin{align}\label{eq-def-p-q}
p = \sum_{i=0}^{\ell-1}P_i,\quad q=\sum_{i=0}^{\ell-n+2}Q_i, 
\end{align}
where $P_0,\ldots,P_{\ell-1}$ and $Q_0,\ldots,Q_{\ell-n+1}$ are known and $Q_{\ell-n+2}$ is to be determined.

Let $w$ be as in \eqref{eq-def-w}. 
Then, $w$ satisfies \eqref{eq:w_eq} and 
$$w=r_{\ell-1}-|y|^{n-2}Q_{\ell-n+2}.$$
We point out that both $\mathcal F$ and $\mathcal G$ in \eqref{eq:w_eq} involve second derivatives of $w$. 
We can write \eqref{eq:w_eq} as
\begin{equation}\label{eq:w_eq_new}\Delta w + |y|^{2n-4}\mathcal F_1(w;y,p,q)
+ |y|^{n-4}\mathcal G_1(w;y,p,q) + h(y) = 0,\end{equation}
where 
\begin{align*}
\mathcal F_1(w;y,p,q) &= \mathcal F(w;y,p,q) - \mathcal F(0;y,p,q),\\
\mathcal G_1(w;y,p,q) &= \mathcal G(w;y,p,q) - \mathcal G(0;y,p,q),
\end{align*} 
and 
\[h(y) = \Delta p + |y|^{2n-4} \mathcal F(0;y,p,q)
+ \Delta( |y|^{n-2}q)+ |y|^{n-4} \mathcal G(0;y,p,q).\]
By \eqref{eq-decom-F-G-0}, we can write 
\[h(y)= \Delta p + \sum_{k\geq0}\widetilde P_k+ \Delta( |y|^{n-2}q)+ |y|^{n-4}\sum_{k\geq0}\widetilde Q_k.\]
The first two terms in the expression of $h$ are polynomials of $y$. 
We consider the last two terms in $h$ and write
\[\Delta( |y|^{n-2}q) + |y|^{n-4} \sum_{k\geq0}\widetilde Q_k
=\sum_{j=0}^{\ell-n+2}\Delta(|y|^{n-2}Q_j) + |y|^{n-4}\sum_{j=0}^{\mu_\ell}\widetilde Q_j,\]
for some positive integer $\mu_\ell\ge \ell-n+2$. 
By Lemma \ref{lm:taylor_expns}, there exists a homogeneous polynomial $Q_{\ell-n+2}$ of degree $\ell-n+2$ such that
\[\Delta(|y|^{n-2}Q_{\ell-n+2}) + |y|^{n-4} \widetilde Q_{\ell-n+2}=0.\]
In fact, $Q_0,\ldots,Q_{\ell-n+1}$ are all determined this way. 
Once we have determined $Q_{\ell-n+2}$,  we can write
\[\Delta(|y|^{n-2}q) + |y|^{n-4}\sum_{k\geq0}\widetilde Q_k = |y|^{n-4}\sum_{j= \ell-n+3}^{\mu_\ell}\widetilde Q_j.\]
This is a $C^{\ell-2,1}$ function of $y$. 
Thus,  $h\in C^{\ell-2,1}(B_1)$ and $p$ and $q$ are known once $Q_{\ell-n+2}$ is determined. 

By our induction hypothesis, we have $w\in C^{\ell-1,\alpha}(B_1)$. 
We will prove $w\in C^{\ell,\alpha}(B_1)$. 
Take an arbitrary $\gamma\in\mathbb Z_+^n$ with $|\gamma|=\ell-2$ and consider $\partial^\gamma w$. 
Since $w$ satisfies \eqref{eq:w_eq_new}, we know that $\partial^\gamma w$ satisfies an equation of the form
\begin{equation}\label{eq:aij_w}\Delta(\partial^\gamma w)
+a_{ij}(\partial^\gamma w)_{ij} + W + \partial^{\gamma}h=0,\end{equation}
where $a_{ij}$ and $W$ are polynomials in terms of $w$ and its derivatives up to order $\ell-1$ 
with coefficients given by polynomials in $y$ and the derivatives of $|y|^{n-4}$. 
We claim 
\begin{align*}a_{ij},W\in C^\alpha(B_1) \quad\text{and}\quad  a_{ij}(0)=0.\end{align*} 
Let us show that this is true for parts of $a_{ij}$ and $W$ which come from $|y|^{n-4}\mathcal G_1(w;y,p,q)$. 
For an illustration, we consider a linear term of the second derivative, 
say $|y|^{n-4}\hat a_{ij}(y)w_{ij}$ for some smooth $\hat a_{ij}$. 
Then,
\begin{equation}\label{eq:w_derivatives}\partial^\gamma(|y|^{n-4}\hat a_{ij}w_{ij}) = |y|^{n-4}\hat a_{ij}(\partial^\gamma w)_{ij} 
+ \sum_{\substack{k+m\leq \ell-2\\ m+2\leq \ell-1}} f_{\beta_1\beta_2}\partial^{\beta_1}(|y|^{n-4})\partial^{\beta_2}w,\end{equation}
where $f_{\beta_1\beta_2}$ are some smooth functions 
and $\beta_1,\beta_2$ are multi-indices of order $k$ and $m+2$ respectively. 
The coefficient $|y|^{n-4}\hat a_{ij}$ in the first term in the right-hand side of \eqref{eq:w_derivatives} 
is a part of $a_{ij}$ from \eqref{eq:aij_w}. 
To see that $|y|^{n-4}\hat a_{ij}\in C^{0,1}(B_1)$ and that $|y|^{n-4}\hat a_{ij}$ vanishes at 0, 
we recall the structure of $\hat a_{ij}$. 
In fact, $\hat a_{ij}$ has the form $y_iy_jF_{ij}$, 
where $F_{ij}$ is smooth. It is clear now that $|y|^{n-4}\hat a_{ij}$ is Lipschitz in $B_1$ 
and vanishes at 0. In particular, $|y|^{n-4}\hat a_{ij}\in C^\alpha(B_1)$ for any $\alpha\in(0,1)$. 
The entire summation in \eqref{eq:w_derivatives} is a part of $W$ and it remains to check 
that it is $C^\alpha$ in $B_1$ for any $\alpha\in(0,1)$. 
This is clear because $w$ vanishes to order $\ell-1$ at 0 by the assumption 
and the terms of $f_{\beta_1\beta_2}\partial^{\beta_1}(|y|^{n-4})\partial^{\beta_2}w$ will have homogeneous degree at least 1 
(this follows from the expression for $\mathcal G$ in \eqref{eq:w_eq}). 
We can analyze other terms in $|y|^{n-4}\mathcal G_1(w;y,p,q)$ and each term in $|y|^{2n-4}\mathcal F_1(w;y,p,q)$
similarly. 

In summary, $\partial^\gamma w$ satisfies a uniformly elliptic equation near the origin with $C^\alpha$ coefficients, 
and so by the Schauder theory, $\partial^\gamma w\in C^{2,\alpha}(B_1)$. 
Since $\gamma$ is an arbitrary multi-index of order $\ell-2$, 
it follows that $w\in C^{\ell,\alpha}(B_1)$ 
and we can write 
$$w=P_\ell+r_\ell,$$ 
where $P_\ell$ is the homogeneous degree-$\ell$ part 
of the Taylor expansion of $w$ and $r_\ell$ is the remainder. 
Hence, 
$$r_\ell\in C^{\ell,\alpha}(B_1)\quad{and}\quad\nabla^kr_\ell(0)=0\text{ for }k=0,\ldots,\ell.$$ 
By setting $p_\ell=p_{\ell-1}+P_\ell$, we have
\[v = p_\ell + |y|^{n-2}q_{\ell-n+2} + r_\ell,\]
where $p_\ell$ is a polynomial of degree $\ell$ and $q_{\ell-n+2}$ is a polynomial of degree $\ell-n+2$, 
with $q_{\ell-n+2}(0)=0$ and $\nabla q_{\ell-n+2}(0)=0$. 
This finishes the proof by induction. 

Last, setting $v_\ell = p_\ell+r_\ell$ yields the desired result.
\end{proof}

We are ready to prove Theorem \ref{new-slag-odd-dim}. 

\begin{proof}[Proof of Theorem \ref{new-slag-odd-dim}]
The proof now follows from Lemma \ref{lm:taylor_expns} in the same way 
that the proof of Theorem 1.3 in \cite{HanWang} follows from Lemma 3.3 in that paper. 
We present the argument here for the sake of completeness.

Let $\alpha\in(0,1)$ be a constant. 
By Lemma \ref{lm:taylor_expns}, there exist sequences of homogeneous polynomials $\{P_i\}_{i\geq0}$ and $\{Q_i\}_{i\geq0}$ 
indexed by degree,  with $Q_0=Q_1=0$, 
and a sequence of functions $\{R_\ell\}_{\ell\geq n-2}$ such that for any $\ell\geq n-2$, we have
\begin{equation}\label{eq:taylor}v = \sum_{i=0}^\ell P_i+|y|^{n-2}\sum_{i=0}^{\ell-n+2}Q_i+R_\ell,\end{equation}
and $R_\ell\in C^{\ell,\alpha}(B_1)$ for any $\alpha\in (0,1)$.

Let $\eta\in C^\infty_0(B_1)$ be a cutoff function with $\eta\equiv1$ in $B_{1/2}$. 
For a sequence of real numbers $\{\lambda_i\}_{i\geq0}$ with $\lambda_i\geq1$ and any $y\in B_1$, set
\[v_1(y)=\sum_{i=0}^\infty P_i(y)\eta(\lambda_iy),\]
and 
\[v_2(y)=\sum_{i=0}^\infty Q_i(y)\eta(\lambda_iy).\]
We claim that we can choose the sequence $\{\lambda_i\}_{i\geq0}$  so that 
$$v_1,v_2, v-v_1-|y|^{n-2}v_2\in C^\infty(B_1).$$
Then, 
$$v=v_1+(v-v_1-|y|^{n-2}v_2)+|y|^{n-2}v_2,$$ 
for $v_1, v-v_1-|y|^{n-2}v_2, v_2\in C^\infty(B_1)$.
This is the desired decomposition for $v$.

To prove our claim, pick $\lambda\geq1$. Then $\eta(\lambda x)=0$ for $|x|\geq1/\lambda$. 
For a homogeneous polynomial $P_i$ of degree $i\ge 1$ 
and any multi-index $\beta\in\mathbb Z^n_+$ with $|\beta|\leq i-1$,  we have
\[\left|\partial^\beta(P_i(y)\eta(\lambda y))\right|\leq\frac{c_{i,\beta}}{\lambda^{i-|\beta|}},\]
for any $y\in B_1$ and some positive constant $c_{i,\beta}$ depending on $P_i$ and $\beta$. 
Hence, for each $i\geq1$, we can pick $\lambda_i$ so that
\[|P_i\eta(\lambda_i\cdot)|_{C^{i-1}(B_1)}\leq 2^{-i},\]
and we can further choose $\lambda_i$ in such a way
 that the resulting sequence $\{\lambda_i\}_{i\geq0}$ is an increasing sequence which tends to $+\infty$ and $\lambda_0=1$. 
 Note that $P_i\eta(\lambda_i\cdot)\in C^\infty(B_1)$ for any $i\geq0$.

For an arbitrary integer $\ell\geq0$, write
\[v_1(y) = \sum_{i=0}^\ell P_i(y)\eta(\lambda_iy)+\sum_{i= \ell+1}^\infty P_i(y)\eta(\lambda_iy).\]
Then, 
\[\sum_{i= \ell+1}^\infty |P_i\eta(\lambda_i\cdot)|_{C^\ell(B_1)}\leq\sum_{i= \ell+1}^\infty |P_i\eta(\lambda_i\cdot)|_{C^{i-1}(B_1)}
\leq\sum_{i= \ell+1}^\infty 2^{-i}=2^{-\ell},\]
and thus 
$$\sum_{i= \ell+1}^\infty P_i\eta(\lambda_i\cdot)\in C^\ell(B_1).$$ 
This shows that $v_1\in C^\ell(B_1)$ and hence $v_1\in C^\infty(B_1)$ since $\ell$ is arbitrary. 
We can prove that $v_2\in C^\infty(B_1)$ in the same way. 
Taking larger $\lambda_i$ if necessary, we can also prove, for any $\ell\geq n-2$,
\[|y|^{n-2}\sum_{i= \ell-n+3}^\infty Q_i\eta(\lambda_iy)\in C^\ell(B_1).\]

Now, for any $\ell\geq n-2$, we write
\[v_1(y)=\sum_{i=0}^\ell P_i(y)+\sum_{i=0}^\ell P_i(y)[\eta(\lambda_iy)-1]+\sum_{i= \ell+1}^\infty P_i(y)\eta(\lambda_iy),\]
and
\[v_2(y)=\sum_{i=0}^{\ell-n+2}Q_i(y)+\sum_{i=0}^{\ell-n+2} Q_i(y)[\eta(\lambda_iy)-1]
+\sum_{i= \ell-n+3}^\infty Q_i(y)\eta(\lambda_iy).\]
By \eqref{eq:taylor}, we have
\[\begin{split}v(y)-v_1(y)-|y|^{n-2}v_2(y)
&=R_\ell(y)+\sum_{i=0}^\ell P_i(y)[\eta(\lambda_iy)-1]\\
&\qquad+|y|^{n-2}\sum_{i=0}^{\ell-n+2} Q_i(y)[1-\eta(\lambda_iy)]\\
&\qquad-\sum_{i= \ell+1}^\infty P_i(y)\eta(\lambda_iy)-|y|^{n-2}\sum_{i= \ell-n+3}^\infty Q_i(y)\eta(\lambda_iy).\end{split}\]
By $\eta\equiv1$ in $B_{1/2}$, every term of $|y|^{n-2}\sum_{i=0}^{\ell-n+2} Q_i[1-\eta(\lambda_iy)]$ 
vanishes near the origin and hence this term is smooth in $B_1$. 
Since $R_\ell\in C^\ell(B_1)$ and we have shown that the remaining terms all have regularity $C^\ell$ in $B_1$, 
we can conclude that $v-v_1-|y|^{n-2}v_2\in C^\ell(B_1)$. 
This implies $v-v_1-|y|^{n-2}v_2\in C^\infty(B_1)$ since $\ell\geq0$ is arbitrary.
\end{proof}

\section{Equations of Special Lagrangian Type}\label{sec-slag-type} 

Yuan observed that the Monge-Amp\`ere equation $\det\nabla^2 u=c$, for a positive constant $c$,  
belongs to a larger family of special Lagrangian equations and suggested that it is possible to find a calibration associated to it.
This was carried out by Warren \cite{Warren2010}. 
There was an earlier related result independently by Mealy \cite{Mealy1991}. 

Specifically, Warren \cite{Warren2010} studied the equation of the form 
\begin{equation}\label{eq-slag-type} F_\tau(\nabla^2u)=\theta,\end{equation} 
in domains in $\mathbb R^n$ for $\tau\in[0,\pi/2]$ and $\theta\in\mathbb R$, where $F_\tau$ is given by
\begin{gather}\label{eq:slag_type}
F_\tau(\nabla^2u) =\begin{cases}
\frac{1}{n}\sum_{j=1}^n\log\lambda_j(\nabla^2u)&\tau=0,\\
\\
\frac{\sqrt{a^2+1}}{2b}\sum_{j=1}^n\log\left(\frac{\lambda_j(\nabla^2u)+a-b}{\lambda_j(\nabla^2u)+a+b}\right)
&\tau\in(0,\pi/4),\\
\\
-\sqrt{2}\sum_{j=1}^n\frac{1}{1+\lambda_j(\nabla^2u)}&\tau=\pi/4,\\
\\
\frac{\sqrt{a^2+1}}{b}\sum_{j=1}^n\arctan\left(\frac{\lambda_j(\nabla^2u)+a-b}{\lambda_j(\nabla^2u)+a+b}\right)&\tau\in(\pi/4,\pi/2),\\
\\
\sum_{j=1}^n\arctan\left(\lambda_j(\nabla^2u)\right)&\tau=\pi/2.
\end{cases}\end{gather}
Here, $a=\cot(\tau)$, $b=\sqrt{|\cot^2(\tau)-1|}$, and $\lambda_j(\nabla^2u)$ denotes the $j$-th eigenvalue of $\nabla^2u$.

The geometric meaning of solutions of the equations in the family \eqref{eq:slag_type} 
is that their gradient graphs are either volume-minimizing or volume-maximizing 
among all homologous $C^1$ space-like surfaces of codimension $n$ in $(\mathbb R^n\times\mathbb R^n,g_\tau)$ where 
\[g_\tau = \cos(\tau)g_0+\sin(\tau)g_{\text{Eucl}}\]
is a linear combination of the standard Euclidean metric
\[g_{\text{Eucl}}=\sum_{i=1}^ndx_i\otimes dx_i+\sum_{j=1}^ndy_j\otimes dy_j\]
and the pseudo-Euclidean metric
\[g_0=\sum_{i=1}^ndx_i\otimes dy_i+\sum_{j=1}^ndy_j\otimes dx_j.\]
Warren \cite{Warren2010} showed that the gradient graphs of solutions of the equations in the family \eqref{eq:slag_type} 
are volume-minimizing among all homologous $C^1$ space-like surfaces of codimension $n$ in 
$(\mathbb R^n\times\mathbb R^n,g_\tau)$ for $\tau\in(\pi/4,\pi/2)$ and 
that the gradient graphs of solutions of the equations in the family are volume-maximizing 
among all homologous $C^1$ space-like surfaces of codimension $n$ 
in $(\mathbb R^n\times\mathbb R^n,g_\tau)$ for $\tau\in[0,\pi/4)$. 
When $\tau=\pi/4$, Warren also showed that the volume of the gradient graph of a solution of $
F_{\pi/4}(\nabla^2u)=\theta$ is equal to the volume of any homologous $C^1$ space-like surface of codimension $n$ 
in $(\mathbb R^n\times\mathbb R^n,g_{\pi/4})$.

We note that $\tau=0$ and $\tau=\pi/2$ correspond to the Monge-Amp\`ere equation and the special Lagrangian equation, 
respectively. In this section, we discuss \eqref{eq-slag-type} for $\tau\in (0, \pi/2)$. 

For $n\ge 3$ and $\tau\in (0, \pi/2)$, Bao and Liu \cite{Liu2022} studied asymptotic behaviors of solutions near infinity. 
Suppose $u\in C^2(\mathbb R^n\setminus\bar B_1)$ is a classical solution of \eqref{eq-slag-type} for $\tau\in (0, \pi/2)$. 
Under appropriate assumptions on $u$ (to be specified later), Bao and Liu \cite{Liu2022} proved that 
there exist an $n\times n$ symmetric matrix $A$, a vector $b\in\mathbb R^n$, and a constant $c\in\mathbb R$ such that, 
for any $x\in \mathbb R^n\setminus\bar B_1$ and any integer $\ell\ge 0$, 
\begin{equation}\label{eq-expansion-slag-type}u(x) = \frac{1}{2}x^TAx+ b\cdot x+c + O_\ell(|x|^{2-n})
\quad\text{as }x\to\infty.\end{equation}
Here, the notation $O_\ell(|x|^{2-n})$ means that
\begin{equation}\label{eq:O_ell}
\limsup_{x\to\infty}|x|^{n-2+\ell}\Big|\nabla^\ell\Big[u(x)-\frac{1}{2}x^TAx+b\cdot x + c\Big]\Big|<\infty.\end{equation}
The expansion \eqref{eq-expansion-slag-type} will serve as the starting point for our study of \eqref{eq-slag-type}. 
As mentioned earlier, the expansion \eqref{eq-expansion-slag-type} was proved by 
Caffarelli and Li \cite{CaLi2003} for the Monge-Amp\`ere equation ($\tau=0$) 
and by Li, Li, and Yuan \cite{LLY2020} for the special Lagrangian equation ($\tau=\pi/2$), respectively. 
Bao and Liu \cite{Liu2022} actually proved a slightly stronger result 
and obtained higher order expansions of the $O_\ell(|x|^{2-n})$-error term, 
as well as expansions up to arbitrary order for radially-symmetric solutions. 
Our result, Theorem \ref{thm:expansions_slag_type}, does not require any assumptions 
of radial symmetry and is proven in a completely different way. 

\begin{theorem}\label{thm:expansions_slag_type}
Let $n\geq3, \tau\in(0,\pi/2)$, and $\theta\in\mathbb R$. Suppose $u\in C^2(\mathbb R^n\setminus\bar B_1)$ 
is a classical solution of \eqref{eq-slag-type}, with 
$F_\tau$ given as in \eqref{eq:slag_type}, and suppose that either of the following holds:
\begin{enumerate}\renewcommand\labelenumi{\normalfont(\roman{enumi})}
\item $\nabla^2u>(-a+b)I$ for $\tau\in(0,\pi/4)$;
\item $\nabla^2u>-I$ for $\tau=\pi/4$;
\item either
\[\nabla^2u > -(a+b)I\quad \text{and}\quad 
\nabla^2u\geq\begin{cases}-(a+bK)I&n\leq4\\-a(\frac{b}{\sqrt{3}}+b\epsilon(n))I&n\geq5\end{cases}\]
for some constants $K$ and $\epsilon$,  or 
\[\nabla^2u>-(a+b)I\quad\text{and}\quad\left|\frac{b\theta}{\sqrt{a^2+1}}+\frac{n\pi}{4}\right|>\frac{n-2}{2}\pi\]
for $\tau\in(\pi/4,\pi/2)$;
\end{enumerate}
Then, there exist an $n\times n$ symmetric matrix $A$, a vector $b\in\mathbb R^n$, a constant $c\in\mathbb R$, 
and a positive definite matrix $R$ such that
\[u(x)=\frac{1}{2}x^TAx+ b\cdot x+c + |R x|^{2-n}v\Big(\frac{R x}{|R x|^2}\Big).\]
Moreover, $v\in C^\infty(B_1)$ when $n$ is even and $v\in C^{n-1,\alpha}(B_1)$ for any $\alpha\in(0,1)$ when $n$ is odd.
\end{theorem}

For solutions of equations in the family \eqref{eq:slag_type}, 
the analog of Theorem \ref{thm:expansions_slag_type} for the Monge-Amp\`ere equation is Theorem 1.2 in \cite{HanWang}
and that for the special Lagrangian equation is Theorem \ref{thrm-slag-expansion} in this paper. 

In Theorem \ref{thm:expansions_slag_type}, the matrix $R$ is determined by $A$, $a$, and $b$. In fact, we have 
\begin{equation}\label{eq:Q_tau}
R =\begin{cases}
(a^2+1)^{-1/4}\big((A+aI)^2-b^2I\big)^{1/2}&\text{for }\tau\in(0,\pi/4),\\
2^{-1/4}(I+A)&\text{for }\tau=\pi/4,\\
(a^2+1)^{-1/4}\big((A+aI)^2+b^2I\big)^{1/2}&\text{for }\tau\in(\pi/4,\pi/2).
\end{cases}\end{equation}
We point out that the assumptions on $\nabla^2u$ in Theorem \ref{thm:expansions_slag_type} 
transform to similar assumptions on the matrix $A$. Specifically, we have 
\begin{equation*}
A >\begin{cases}
(-a+b)I&\text{for }\tau\in(0,\pi/4),\\
-I&\text{for }\tau=\pi/4,\\
-(a+b)I&\text{for }\tau\in(\pi/4,\pi/2).
\end{cases}\end{equation*}
We also have $b>0$ for $\tau\in(0,\pi/4)$ and $0<b<1$ for $\tau\in(\pi/4,\pi/2)$.
It is straightforward to check that $(A+aI)^2-b^2I$, $I+A$, and $(A+aI)^2+b^2I$ are positive definite 
for $\tau\in(0,\pi/4)$, $\tau=\pi/4$, and $\tau\in(\pi/4,\pi/2)$, respectively. 

For odd dimensions, we prove that the function $v$ in Theorem \ref{thm:expansions_slag_type} 
can  be decomposed into the sum of a smooth function and the product of a $C^{n-1,1}$-factor and a smooth function.

\begin{theorem}\label{thm:expansions_odd_n}
Let $v$ be as in theorem \ref{thm:expansions_slag_type} for $n\geq 3$ odd. 
Then there exist $v_1,v_2\in C^\infty(B_1)$ with $v_2(0)=|\nabla v_2(0)|=0$ such that
\[v(y) = v_1(y)+|y|^{n-2}v_2(y).\]
\end{theorem}

We now start to prove Theorem \ref{thm:expansions_slag_type}. 
The proof follows a similar structure as that of Theorem \ref{thrm-slag-expansion}. 
The key idea is to rewrite the solutions to \eqref{eq-slag-type} in $\mathbb R^n\setminus\bar B_1$ 
as solutions of a perturbed Laplace equation in $B_1\setminus\{0\}$ 
using a modified Kelvin transform.

\begin{proof}[Proof of Theorem \ref{thm:expansions_slag_type}] 
We will adopt the same notations in the proof of Theorem \ref{thrm-slag-expansion}. 
Let $A$ be the diagonal matrix as in \eqref{eq-expression-A-slag}. 
For $y$ in \eqref{eq-y-slag}, we introduce a function $v=v(y)$ in $B_1\setminus\{0\}$ such that 
\eqref{eq:special_lagrangian_kelvin} holds.
In the proof below, we discuss the three cases 
$\tau\in(\pi/4,\pi/2)$, $\tau=\pi/4$, $\tau\in(0, \pi/4)$ separately and choose the matrix $R$ in \eqref{eq-y-slag} 
in these three cases. We will prove that $v$ satisfies an equation similar as \eqref{eq:special_lagrangian_v-dim-n-v2}. 
In the present proof, we concentrate only on the linear forms in $v$ and its derivatives. 
The study of higher order terms is similar as in the proof of Theorem \ref{thrm-slag-expansion} and is omitted. 

\textit{Case 1: $\tau\in(\pi/4,\pi/2)$.}
We take 
\begin{equation}\label{eq-choice-R-case1}
R=(a^2+1)^{-1/4}\diag\big(((\lambda_1+a)^2+b^2)^{1/2},\ldots,((\lambda_n+a)^2+b^2)^{1/2}\big).\end{equation}
We first claim that
\[\frac{\sqrt{a^2+1}}{b}\sum_{j=1}^n\arctan\Big(\frac{\lambda_j(\nabla^2u)+a-b}{\lambda_j(\nabla^2u)+a+b}\Big)=\theta\]
is equivalent to
\begin{equation}\label{eq:algebraic_tau_pi/4_pi/2}\begin{split}
&\sigma_n(\lambda_j(\nabla^2u+(a+b)I))
\Big[\sin(\tilde\theta)\sum_{0\leq2k\leq n}(-1)^k\sigma_{2k}\Big(\frac{\lambda_j(\nabla^2u+(a-b)I)}{\lambda_j(\nabla^2u+(a+b)I)}\Big)
\\&\qquad-\cos(\tilde\theta)\sum_{1\leq 2k+1\leq n}(-1)^k\sigma_{2k+1}
\Big(\frac{\lambda_j(\nabla^2u+(a-b)I)}{\lambda_j(\nabla^2u+(a+b)I)}\Big)\Big]=0,\end{split}\end{equation}
where $\tilde\theta=\frac{\theta b}{\sqrt{a^2+1}}$. 
The proof of this claim is similar to the derivation of the algebraic form of the special Lagrangian equation. 
We present it here for the sake of completeness.
Set
\[\mu_j=\frac{\lambda_j(\nabla^2u+aI)-b}{\lambda_j(\nabla^2u+aI)+b}.\]
Then,
\begin{equation}\label{eq:muj}\begin{split}
\sum_{j=1}^n\arctan(\mu_j)
&=\arctan\bigg(\frac{\sum_{1\leq 2k+1\leq n}(-1)^k\sigma_{2k+1}(\mu_j)}{\sum_{0\leq2k\leq n}(-1)^k\sigma_{2k}(\mu_j)}\bigg)\\
&=\arctan\bigg(\frac{\sum_{1\leq 2k+1\leq n}(-1)^k\sigma_{2k+1}
\Big(\frac{\lambda_j(\nabla^2u+aI)-b}{\lambda_j(\nabla^2u+aI)+b}\Big)}
{\sum_{0\leq2k\leq n}(-1)^k\sigma_{2k}\Big(\frac{\lambda_j(\nabla^2u+aI)-b}{\lambda_j(\nabla^2u+aI)+b}\Big)}\bigg).
\end{split}\end{equation}
Now, we write the right-hand side of \eqref{eq:muj} as
\[\arctan\bigg(\frac{\sigma_n(\lambda_j(\nabla^2u+aI)+b)\sum_{1\leq 2k+1\leq n}(-1)^k\sigma_{2k+1}
\Big(\frac{\lambda_j(\nabla^2u+aI)-b}{\lambda_j(\nabla^2u+aI)+b}\Big)}
{\sigma_n(\lambda_j(\nabla^2u+aI)+b)\sum_{0\leq2k\leq n}(-1)^k\sigma_{2k}
\Big(\frac{\lambda_j(\nabla^2u+aI)-b}{\lambda_j(\nabla^2u+aI)+b}\Big)}\bigg).\]
Doing this clears the denominators in the expressions of the form
\[\sigma_{k}\Big(\frac{\lambda_j(\nabla^2u+aI)-b}{\lambda_j(\nabla^2u+aI)+b}\Big),\]
and turns the expression inside of the $\arctan$ in \eqref{eq:muj} into a ratio of degree-$n$ polynomials. 
Thus, we have
\[\tan(\tilde\theta)
=\frac{\sigma_n(\lambda_j(\nabla^2u+aI)+b)\sum_{1\leq 2k+1\leq n}(-1)^k\sigma_{2k+1}
\Big(\frac{\lambda_j(\nabla^2u+aI)-b}{\lambda_j(\nabla^2u+aI)+b}\Big)}
{\sigma_n(\lambda_j(\nabla^2u+aI)+b)\sum_{0\leq2k\leq n}(-1)^k\sigma_{2k}
\Big(\frac{\lambda_j(\nabla^2u+aI)-b}{\lambda_j(\nabla^2u+aI)+b}\Big)},\]
which yields \eqref{eq:algebraic_tau_pi/4_pi/2}.

Now, using the $\bar\sigma$ notation we introduced in \eqref{eq:sigma_bar_def}, 
we rewrite \eqref{eq:algebraic_tau_pi/4_pi/2} as
\[\sin(\tilde\theta)\sum_{0\leq2k\leq n}(-1)^k\bar\sigma_{2k}\(\nabla^2u\)
-\cos(\tilde\theta)\sum_{1\leq 2k+1\leq n}(-1)^k\bar\sigma_{2k+1}\(\nabla^2u\)=0.\]
Since $u$ satisfies \eqref{eq-expansion-slag-type}, we conclude that
\[\sin(\tilde\theta)\sum_{0\leq2k\leq n}(-1)^k\bar\sigma_{2k}(A)
-\cos(\tilde\theta)\sum_{1\leq 2k+1\leq n}(-1)^k\bar\sigma_{2k+1}(A)=0.\]
By eliminating $\tilde\theta$ and making the substitution $\nabla^2u=A+|y|^nN$, 
we see that \eqref{eq:algebraic_tau_pi/4_pi/2} is equivalent to
\begin{equation}\label{eq:algebraic_tau_pi/4_pi/2_no_theta}\begin{split}
&\Big(\sum_{0\leq 2k\leq n}(-1)^k\bar\sigma_{2k}(A)\Big)
\Big(\sum_{1\leq 2k+1\leq n}(-1)^k\bar\sigma_{2k+1}(A+|y|^nN)\Big) 
\\&\qquad
-\Big(\sum_{1\leq 2k+1\leq n}(-1)^k\bar\sigma_{2k+1}(A)\Big)
\Big(\sum_{0\leq 2k\leq n}(-1)^k\bar\sigma_{2k}(A+|y|^nN)\Big) = 0.
\end{split}\end{equation}

We now arrange \eqref{eq:algebraic_tau_pi/4_pi/2_no_theta} according to the powers of $N_{ij}$. 
First, the zero-order terms all cancel out. Next, we examine the linear terms in $N_{ij}$. 
By Lemma \ref{lm:sigma_k_to_sigma_hat} and Lemma \ref{lm:sigma_bar_to_sigma_k}, we see that 
all linear terms are given by 
\[\begin{split}
&\sum_{i=1}^n\Big[\Big(\sum_{0\leq 2k\leq n}(-1)^k\bar\sigma_{2k}(A)\Big)\\
&\qquad\cdot\sum_{1\leq 2k+1\leq n}(-1)^k\Big(\sum_{l=0}^n\sum_{j=0}^l
\binom{n-l}{2k+1-j}(a-b)^{2k+1-j}(a+b)^{n-l-2k-1+j}\hat \sigma_{l-1,i}(A)\Big)\\
&\qquad-\Big(\sum_{1\leq 2k+1\leq n}(-1)^k\bar\sigma_{2k+1}(A)\Big)\\
&\qquad\cdot\sum_{0\leq 2k\leq n}(-1)^k\Big(\sum_{l=0}^n\sum_{j=0}^l
\binom{n-l}{2k-j}(a-b)^{2k-j}(a+b)^{n-l-2k+j}\hat \sigma_{l-1,i}(A)\Big)\Big]N_{ii}.\end{split}\]
Lemma \ref{lm:linear_tau_pi/4_pi/2} yields that this is 
\[\begin{split}
2^nb&\sum_{i=1}^n\big[(\lambda_1+a)^2+b^2)\cdot\ldots\cdot((\lambda_{i-1}+a)^2+b^2)\\
&\qquad\qquad\cdot((\lambda_{i+1}+a)^2+b^2)\cdot\ldots\cdot((\lambda_n+a)^2+b^2)\big]N_{ii}.\end{split}\]
By \eqref{eq:mij_tilde-v2} and \eqref{eq-choice-R-case1}, we have 
$$N_{ii}=R_{ii}^2M_{ii}=((\lambda_i+a)^2+b^2)M_{ii}.$$
Hence, all linear terms are given by 
\begin{align*}
2^nb((\lambda_1+a)^2+b^2)\cdot\ldots\cdot((\lambda_n+a)^2+b^2)|y|^{n}\Tr(M)=
2^nb\det(R^2)|y|^{n+2}\Delta v,\end{align*}
where we used \eqref{eq-trace-M}. 
Note that $0<b<1$ for $\tau\in(\pi/4, \pi/2)$. 

Observe now that Lemma \ref{lm:sigma_bar_to_sigma_k} allows us 
to rewrite \eqref{eq:algebraic_tau_pi/4_pi/2_no_theta} as a (very messy) equation 
in terms of only $\sigma_k(N)$. 
Our analysis of the potentially singular terms in the case of the special Lagrangian equation 
depends only on the precise structure of $N$ and $\sigma_k(N)$, 
but \textit{not} on the specific structure of the special Lagrangian equation itself 
(beyond showing that the $N_{ij}$-linear terms simplify to $\Delta v$). 
Therefore, it follows  that we may rewrite \eqref{eq:algebraic_tau_pi/4_pi/2_no_theta} 
in the form \eqref{eq:special_lagrangian_v-dim-n-v2}. 
From here, we proceed exactly as in the case of the special Lagrangian equation to prove the higher regularity of $v$.

\textit{Case 2:  $\tau=\pi/4$.} We take 
\begin{equation}\label{eq-choice-R-case2}R=2^{-1/4}\diag(1+\lambda_1,\ldots,1+\lambda_n).\end{equation}
Using the fact that
\[\sum_{j=1}^n\frac{1}{1+\lambda_j(\nabla^2u)} = \frac{\sigma_{n-1}(I+\nabla^2u)}{\det(I+\nabla^2u)},\]
it is not hard to see that 
\[-\sqrt{2}\sum_{j=1}^n\frac{1}{1+\lambda_j(\nabla^2u)}=\theta\]
is equivalent to
\begin{equation}\label{eq:algebraic_tau_pi/4}
-\sqrt{2}\sigma_{n-1}(I+\nabla^2u)-\theta\det(I+\nabla^2u)=0.\end{equation}
Since $u$ satisfies \eqref{eq-expansion-slag-type}, we conclude that
\begin{equation*}
-\sqrt{2}\sigma_{n-1}(I+A)-\theta\det(I+A)=0.\end{equation*}
By eliminating $\theta$, we note that \eqref{eq:algebraic_tau_pi/4} is equivalent to
\[\det(I+A)\sigma_{n-1}(I+\nabla^2u)-\sigma_{n-1}(I+A)\det(I+\nabla^2u)=0.\]
Hence, $v$ solves
\begin{equation}\label{eq:algebraic_tau_pi/4_no_theta}
\det(I+A)\sigma_{n-1}(I+A+|y|^nN)-\sigma_{n-1}(I+A)\det(I+A+|y|^nN)=0\end{equation}
in $B_1\setminus\{0\}$. 

As in Case 1, we analyze the linear expressions of $N_{ij}$. 
We can apply Lemma \ref{lm:sigma_k_to_sigma_hat}, but with the matrix $A+I$ in place of $A$. 
Thus, we see that the linear terms of $\sigma_{n-1}(I+A+|y|^nN(v))$ and $\det(I+A+|y|^nN(v))$
are given by $|y|^n\sum_{i= 1}^n\hat \sigma_{n-2,i}(I+A)N_{ii}$ and
$|y|^n\sum_{i = 1}^n\hat \sigma_{n-1,i}(I+A)N_{ii}$, respectively. 
Then, the linear terms of \eqref{eq:algebraic_tau_pi/4_no_theta} are
\begin{equation}\label{eq:linear_tau_pi/4}
|y|^n\sum_{i=1}^n\big[(1+\lambda_1)\cdot\ldots\cdot(1+\lambda_n)\hat \sigma_{n-2,i}(I+A)
-\sigma_{n-1}(I+A)\hat \sigma_{n-1,i}(I+A)\big]N_{ii}.\end{equation}
We now prove
\begin{equation}\label{eq:claim_lambda_tau_pi/4}
\begin{split}&(1+\lambda_1)\cdot\ldots\cdot(1+\lambda_n)\hat \sigma_{n-2,i}(I+A)
-\sigma_{n-1}(I+A)\hat \sigma_{n-1,i}(I+A)\\
&\qquad=-(1+\lambda_1)^2\cdot\ldots\cdot(1+\lambda_{i-1})^2(1+\lambda_{i+1})^2\cdot\ldots\cdot(1+\lambda_n)^2.
\end{split}\end{equation}
We write the first term in the left-hand side of \eqref{eq:claim_lambda_tau_pi/4} as
\[(1+\lambda_i)\hat\sigma_{n-1,i}(I+A)\hat \sigma_{n-2,i}(I+A),\]
and note that $\sigma_{n-1}(I+A)=(1+\lambda_i)\hat\sigma_{n-2,i}(I+A)+\hat\sigma_{n-1,i}(I+A)$ 
to write the second term as
\[(1+\lambda_i)\hat\sigma_{n-2,l}(I+A)\hat\sigma_{n-1,i}(I+A)+\hat\sigma_{n-1,i}(I+A)\hat\sigma_{n-1,i}(I+A).\]
Now, the left-hand side of \eqref{eq:claim_lambda_tau_pi/4} is simply
\begin{align*}
&(1+\lambda_i)\hat\sigma_{n-1,i}(I+A)\hat \sigma_{n-2,i}(I+A)
-(1+\lambda_i)\hat\sigma_{n-2,i}(I+A)\hat\sigma_{n-1,i}(I+A)\\
&\qquad-\hat\sigma_{n-1,i}(I+A)\hat\sigma_{n-1,i}(I+A),
\end{align*}
which yields \eqref{eq:claim_lambda_tau_pi/4}. 
By \eqref{eq:mij_tilde-v2} and \eqref{eq-choice-R-case2}, we have 
\[N_{ii}=\sqrt{2}(1+\lambda_i)^2M_{ii}.\]
Then, we can write \eqref{eq:linear_tau_pi/4} as 
\begin{align*}-\sqrt{2}(1+\lambda_1)^2\cdot\ldots\cdot(1+\lambda_n)^2|y|^{n+2}\Delta v.\end{align*}

\textit{Case 3: $\tau\in(0,\pi/4)$.} We take 
\begin{equation}\label{eq-choice-R-case3}
R=(a^2+1)^{-1/4}\diag\big(((\lambda_1+a)^2-b^2)^{1/2},\ldots,((\lambda_n+a)^2-b^2)^{1/2}\big).\end{equation}
It is not hard to see that 
\[\frac{\sqrt{a^2+1}}{2b}\sum_{j=1}^n\log\Big(\frac{\lambda_j(\nabla^2u)+a-b}{\lambda_j(\nabla^2u)+a+b}\Big) =\theta\]
is equivalent to
\begin{equation}\label{eq:tau_0_pi/4}
e^{\frac{\sqrt{a^2+1}}{2b}}\det(\nabla^2u+(a-b)I)-e^{\theta}\det(\nabla^2u+(a+b)I)=0.\end{equation}
Since $u$ satisfies \eqref{eq-expansion-slag-type}, we conclude that
\[e^{\frac{\sqrt{a^2+1}}{2b}}\det(A+(a-b)I)-e^{\theta}\det(A+(a+b)I)=0.\]
We eliminate $\theta$ from \eqref{eq:tau_0_pi/4} and consider
\[\begin{split}&\det(A+(a+b)I)\det(\nabla^2u+(a-b)I)
\\&\qquad-\det(A+(a-b)I)\det(\nabla^2u+(a+b)I)=0.\end{split}\]
Hence, $v$ solves
\begin{equation}\label{eq:v_tau_0_pi/4}\begin{split}
&\det(A+(a+b)I)\det(A+(a-b)I+|y|^nN)
\\&\qquad-\det(A+(a-b)I)\det(A+(a+b)I+|y|^nN)=0\end{split}\end{equation}
in $B_1\setminus\{0\}$. 

The linear terms of \eqref{eq:v_tau_0_pi/4} have the form
\[\begin{split}
&|y|^n\Big[(\lambda_1+a+b)\cdot\ldots\cdot(\lambda_n+a+b)\sum_{i = 1}^n\hat \sigma_{n-1,i}(A+(a-b)I)N_{ii}\\
&\qquad-(\lambda_1+a-b)\cdot\ldots\cdot(\lambda_n+a-b)\sum_{i = 1}^n\hat \sigma_{n-1,i}(A+(a+b)I)N_{ii}\Big].\end{split}\]
Consider
\begin{align}\label{eq-difference-0_pi/4}\begin{split}
&(\lambda_1+a+b)\cdot\ldots\cdot(\lambda_n+a+b)\hat\sigma_{n-1,i}(A+(a-b)I)N_{ii}
\\&\qquad-(\lambda_1+a-b)\cdot\ldots\cdot(\lambda_n+a-b)\hat\sigma_{n-1,i}(A+(a+b)I)N_{ii}.\end{split}\end{align}
By \eqref{eq:mij_tilde-v2} and \eqref{eq-choice-R-case3}, we have 
$$N_{ii}=((\lambda_i+a)^2-b^2)M_{ii}.$$ 
Hence,
\[\begin{split}&(\lambda_1+a+b)\cdot\ldots\cdot(\lambda_n+a+b)\hat\sigma_{n-1,i}(A+(a-b)I)N_{ii}\\
&\qquad=(\lambda_i+a+b)(\lambda_1+a)^2-b^2)\cdot\ldots\cdot((\lambda_n+a)^2-b^2)M_{ii},\end{split}\]
and
\[\begin{split}&(\lambda_1+a-b)\cdot\ldots\cdot(\lambda_n+a-b)\hat\sigma_{n-1,i}(A+(a+b)I)N_{ii}
\\&\qquad=(\lambda_i+a-b)(\lambda_1+a)^2-b^2)\cdot\ldots\cdot((\lambda_n+a)^2-b^2)M_{ii}.\end{split}\]
A simple subtraction shows that \eqref{eq-difference-0_pi/4} has the form
\[2b(\lambda_1+a)^2-b^2)\cdot\ldots\cdot((\lambda_n+a)^2-b^2)M_{ii}.\]
Hence, the linear terms of \eqref{eq:v_tau_0_pi/4} are given by 
\[2b((\lambda_1+a)^2-b^2)\cdot\ldots\cdot((\lambda_n+a)^2-b^2)|y|^{n+2}\Delta v
=2^nb\det(R^2)|y|^{n+2}\Delta v.\]
Note that $b>0$ for $\tau\in(0,\pi/4)$. 

In summary, we conclude that for all three cases, the linear part of the equation is a nonzero constant multiple of 
$|y|^{n+2}\Delta v$. 
\end{proof}

The proof of Theorem \ref{thm:expansions_odd_n} is similar to that of Theorem \ref{new-slag-odd-dim} 
and does not depend on the explicit form of the equations we consider. 
We omit the proof of Theorem \ref{thm:expansions_odd_n}.

\end{document}